\documentclass[11pt]{article}

\usepackage[backend=bibtex, doi=false, url=false, isbn=false, style=alphabetic]{biblatex}
\DeclareNameAlias{sortname}{last-first}
\usepackage{fullpage}
\usepackage[T1]{fontenc}
\usepackage{lmodern}
\usepackage{hyperref}
\usepackage{xcolor}
\definecolor{link-color}{rgb}{0.15,0.4,0.15}
\hypersetup{
  colorlinks,
  linkcolor = {link-color}, citecolor = {link-color},
}

\usepackage{graphicx}
\usepackage{caption}
\usepackage{subcaption}
\usepackage{hyperref}
\usepackage{amsmath,amsthm,amssymb}
\usepackage{bbm}
\usepackage{tabularx}

\newtheorem{thm}{Theorem}[section]

\newtheorem{cor}[thm]{Corollary}
\newtheorem{lemma}[thm]{Lemma}
\theoremstyle{plain}
\newtheorem{rem}[thm]{Remark}
\newtheorem{defn}[thm]{Definition}

\newcommand{\R}{\mathbb{R}}
\newcommand{\N}{\mathbb{N}}
\renewcommand{\P}{\mathbb{P}}
\newcommand{\E}{\mathbb{E}}
\newcommand{\1}{\mathbbm{1}}
\newcommand{\F}{\mathcal{F}}
\newenvironment{enum}{\begin{list}{(\roman{enumi})}{\usecounter{enumi}}}{\end{list}}

\newcommand{\dist}{\text{dist}}
\newcommand{\Z}{\mathcal{Z}}

\bibliography{references.bib}

\begin{document}

  %
  %
  %
  %
  %
  %
  %
  %
  %
  %
  %
  %
  %

  \author{Bat\i{} \c{S}eng\"ul\thanks{University of Cambridge, email: batisengul@gmail.com} \thanks{Supported by the UK Engineering and Physical Sciences Research Council (EPSRC) grant EP/H023348/1 for the University of Cambridge Centre for Doctoral Training, the Cambridge Centre for Analysis.}}
  \title{Scaling Limits of Coalescent Processes Near Time Zero}
  \date{\today}

  \maketitle

  \begin{abstract}
    In this paper we obtain scaling limits of $\Lambda$-coalescents near time zero under a regularly varying assumption. In particular this covers the case of Kingman's coalescent and beta coalescents. The limiting processes are coalescents with infinite mass, obtained geometrically as tangent cones of Evans metric space associated with the coalescent. In the case of Kingman's coalescent we are able to obtain a simple construction of the limiting space using a two-sided Brownian motion.
  \end{abstract}

  \section{Introduction}
  \subsection{Statement of the main results}
  A coalescent process is a particle system in which particles merge into blocks. Coalescent processes have found a variety of applications in physics, chemistry and most notably in genetics where the coalescent process models ancestral relationships as time runs backwards. The work on coalescent theory dates back to the seminal paper \cite{kingman} where Kingman considered coalescent processes with pairwise mergers. In \cite{pitmanlambdacoal}, \cite{sagitovlambdacoal} and \cite{dk99} this was extended to the case where multiple mergers are allowed to happen. We refer to \cite{berestyckibook} and \cite{bertoinfragment} for an overview of the field.

  In this paper we shall consider $\Lambda$-coalescents where $\Lambda$ is a finite strongly regularly varying measure with index $1<\alpha\leq 2$, see \eqref{eq:regular_vary}. These coalescents encompass a large variety of well known examples such as beta coalescents and Kingman's coalescent. Further, these coalescents have the property that they come down from infinity, that is, when starting with infinitely many particles, the process has finitely many blocks for any time $t>0$.

  Our goal is to gain precise information about the behaviour of the coalescent near time zero under the regularly varying assumption. Prior work on this appears in \cite{berestyckismall} where the authors establish many results about the behaviour of beta coalescents as $t \downarrow 0$; they show the asymptotic behaviour of the number of blocks and further show results about the largest block and the block containing $1$ at time $t$. These results are generalisations of the results for the case of Kingman's coalescent (see \cite{aldouscoalescence}) and concern the behaviour of a single block at a single time. Our aim here is to describe the behaviour of many blocks spanning over a small interval of time.

  One central insight of this work is that the correct framework for taking such scaling limits is to view coalescent processes as geometric objects. What follows is an outline of our approach. To any coalescent process $\Pi=(\Pi(t): t\geq 0)$, one can associate a certain ultra--metric space $(E,\delta)$ which completely characterises the process $\Pi$. This was first suggested in the work of Evans in \cite{evansmetric}, who introduced this object in the case of Kingman's coalescent and studied some of its properties.

  The construction is simple and we describe it now. Let $\Pi=(\Pi(t):t \geq 0)$ be a coalescent process and define an ultra--metric on $\N$ by
  \begin{equation}\label{eq:evans_metric_def}
    \delta(i,j):=\inf\{t >0: i \overset{\Pi(t)}{\sim} j\}
  \end{equation}
  where $i \overset{\Pi(t)}{\sim} j$ if and only if $i,j$ are in the same block of $\Pi(t)$. The metric space $(E,\delta)$ is then the completion of $(\N,\delta)$.

  Notice that $\delta(i,j)$ gives the time for the most recent common ancestor of $i$ and $j$. Moreover it is not hard to check that $(E,\delta)$ is compact if and only if the coalescent process $\Pi$ comes down from infinity, that is for all $t >0$, $\Pi(t)$ has finitely many (non-empty) blocks. We call the space $(E,\delta)$ the Evans space associated to the coalescent $\Pi$.

  Let $\Lambda$ be a finite measure on $[0,1]$. We say that $\Lambda$ is SRV$(\alpha)$ if $\Lambda$ is strongly regularly varying with index $\alpha\in (1,2)$. That is when $\Lambda(dp)=f(p)\,dp$ and there exists a constant $A_\Lambda>0$ such that
  \begin{equation}\label{eq:regular_vary}
    f(p)\sim A_\Lambda p^{1-\alpha} \qquad\qquad p\rightarrow 0
  \end{equation}
  where the above notation means the quotient of both sides approaches $1$. We will abuse notation slightly and say that $\Lambda$ is SRV$(2)$ when $\Lambda = \delta_{\{0\}}$. It is possible to associate with each finite $\Lambda$ on $[0,1]$ a coalescent process called the $\Lambda$-coalescent and the case when $\Lambda$ is SRV$(2)$, the $\Lambda$-coalescent is Kingman's coalescent. Note that if $\Lambda$ is a finite SRV$(\alpha)$, then the $\Lambda$-coalescent comes down from infinity if and only if $\alpha \in (1,2]$.

  Finite SRV$(\alpha)$ measures encompass a large variety of measures. A prominent example is the Beta$(2-\alpha,\alpha)$ distribution which has density $B(2-\alpha,\alpha)^{-1}p^{1-\alpha}(1-p)^{\alpha -1}\, dp$, where $B(x,y)$ is the beta function. This is a one parameter family which interpolates between the uniform measure ($\alpha =1$) and $\delta_0$ ($\alpha \rightarrow 2$) for which the corresponding coalescents are the Bolthausen-Sznitman coalescent and Kingman's coalescent respectively. The importance of the SRV$(\alpha)$ condition stems primarily from population genetics where the models correspond to populations in which there is large variability in the offspring distribution, see \cite[Section 3.2]{berestyckibook}.

  The first theorem of the paper presented below shows convergence of the metric spaces that correspond to the coalescent processes as in \eqref{eq:evans_metric_def}. Let us briefly discuss the pointed Gromov--Hausdorff topology which we use as our notion of convergence (see Section \ref{sec:preliminary} for the details). A pointed metric space $(S,d,p)$ is called proper if every closed ball is compact and Polish if it is complete and separable. A sequence of proper Polish pointed metric spaces $(S_n,d_n,p_n)$ converges to a proper Polish pointed metric space $(S,d,p)$ under the pointed Gromov--Hausdorff topology if for every $r>0$, the closed ball of radius $r$ around $p_n \in S_n$ converges in the usual Gromov--Hausdorff sense to the closed ball of radius $r$ around $p \in S$. The space of all proper Polish pointed metric spaces can be equipped with a metric, called the pointed Gromov--Hausdorff metric, which is compatible with the notion of convergence described and further this space is itself a Polish space when equipped with the pointed Gromov--Hausdorff metric.

  \begin{thm}\label{thm:main}
    Let $\Lambda$ be a finite measure satisfying \eqref{eq:regular_vary} for some $\alpha \in (1,2]$ and $(E,\delta)$ be the Evans space associated to the corresponding $\Lambda$ coalescent via \eqref{eq:evans_metric_def}. Then for all $i \in \N$, there exists a random pointed  ultra--metric space $(\mathbb S,d_\mathbb S,o)$, which is independent of $i$, such that
    \[
      (E,\epsilon^{-1}\delta(\cdot,\cdot),i) \rightarrow (\mathbb  S, d_\mathbb  S, o)
    \]
    weakly under the pointed Gromov--Hausdorff metric as $\epsilon \rightarrow 0$.
  \end{thm}

  The limiting spaces in the theorem depend on the value of $\alpha \in (1,2]$ and $A_\Lambda$. We will denote them by $(\mathbb  S^{(\alpha)}, d_\mathbb  S^{(\alpha)}, o^{(\alpha)})$ if there is a risk of confusion.

  Geometrically, the space $(\mathbb S, d_\mathbb S, o)$ in Theorem \ref{thm:main} is referred to as tangent cone of $(E,\delta)$ at the point $i$. More precisely a tangent cone of a metric space $(X,d)$ at a point $x \in X$ is given by the pointed Gromov--Hausdorff limit of $(X,r_i^{-1}d,x)$ as $i \rightarrow \infty$ where $\{r_i\}_{i\geq 1}$ is some sequence such that $r_i \downarrow 0$. Tangent cones are generalisations of tangent spaces on manifolds. Indeed, on a Riemannian manifold the tangent cone at any point exists and is isometric to the tangent space. Tangent cones have appeared in a variety of contexts ranging from geometric measure theory \cite{leonsimon} to a recent paper \cite{brownianplane} in which the tangent cones of the Brownian map are identified as the Brownian plane. In our case, tangent cones are the correct objects for describing the scaling limits as they allow us to forget about the mass and ordering imposed on the coalescent.

  In \cite{hughesequivalence} the author identifies a homeomorphism between the space of ultra--metric spaces and the space of real trees both equipped with the Gromov--Hausdorff metric. Consequently Theorem \ref{thm:main} can be stated in terms of the real trees that correspond to the coalescents. The tangent cones are only of interest at the leaves of a real tree as they can be easily identified at any other point as follows. If $(T,d)$ is a coalescent tree and $x \in T$ such that $T\backslash \{x\}$ has exactly two connected components then the tangent cone $\lim_{r \downarrow 0} (T,r^{-1} d, x)$ exists and is isometric to $\R$ with the Euclidean distance. If $T\backslash \{x\}$ has $k \geq 3$ components then the tangent cone around $x$ exists and is isometric to $k$ disjoint copies of $[0, \infty)$ glued together at the point $0$, equipped with the intrinsic metric.

  The next result (which is both a crucial step in the proof of Theorem \ref{thm:main}, and of independent interest) provides a description of the mergers of the block containing $1$ at small times. This description will allow us to depict the space $(\mathbb S,d_\mathbb S)$. Loosely speaking, this result should be interpreted as a local limit of the coalescent tree, whereas Theorem \ref{thm:main} deals with global scaling limits. More precisely for $\epsilon > 0$ and $r \in [0,1)$ let $\Z_\epsilon(r)$ be the number of blocks of $\Pi((1-r)\epsilon)$ that make up $\Pi_1(\epsilon)$, the block containing $1$ at time $\epsilon$ (see also Figure \ref{fig:Z_eps}). Thus there exists $1=i_1<\dots < i_{\Z_\epsilon(r)}$ such that
  \[
    \Pi_1(\epsilon)= \Pi_{i_1}((1-r)\epsilon) \cup \dots \cup\Pi_{i_{\Z_\epsilon(r)}}((1-r)\epsilon).
  \]
  where $\Pi_i(t)$ denotes the $i$--th block of $\Pi(t)$ under the standard infimum ordering.
  Henceforth we shall be considering the c\`adl\`ag modification of the process $\Z_\epsilon$.

  \begin{figure}
    \begin{center}
      \includegraphics[scale=1]{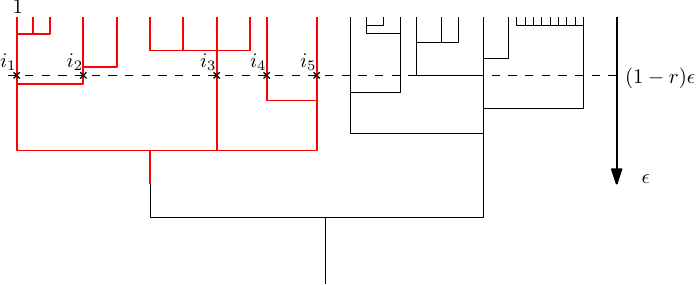}
    \end{center}
    \caption{Picture illustrating the process $\Z_\epsilon$. The red sub-tree represents the blocks of the coalescent process which eventually merge to form $\Pi_1(\epsilon)$. Here $\Z_\epsilon(r)=5$.}\label{fig:Z_eps}
  \end{figure}

  \begin{thm}\label{thm:path_back}
    For $\alpha \in (1,2]$ and $\Lambda$ a finite SRV$(\alpha)$ measure let $\Z_\epsilon$ be the process constructed above using a $\Lambda$-coalescent. Then as $\epsilon \rightarrow 0$, $\Z_\epsilon \rightarrow \Z$ in the Skorokhod sense on $[0,1)$. The process $\Z$ is an inhomogeneous Markov process with generator
    \[
      L_rf(i)=A_\Lambda\sum_{j\geq 1}(i+j)\frac{\Gamma(2-\alpha)\Gamma(j-\alpha+1)}{(1-r)\alpha\Gamma(j+2)} [f(i+j)-f(i)]
    \]
    when $\alpha \in (1,2)$ and
    \[
      L_rf(i)=\frac{(i+1)}{1-r} [f(i+1)-f(i)]
    \]
    when $\alpha =2$.
  \end{thm}

  We will now depict how the closed unit ball $B(o,1) \subset (\mathbb S,d_\mathbb S)$ is constructed (see Section \ref{sec:proof_a_ranged} for a proof of the assertions here). First construct a tree $T$ from a branching process. Start the tree with one particle which does not die and is hence referred to as an immortal particle. The immortal particle produces $j \geq 1$ offspring at time $r \in [0,1)$ with rate $(j+1)q_{j+1}^{(r)}$ where
  \[
    q_{j+1}^{(r)} =\begin{cases}
      \displaystyle\frac{\Gamma(2-\alpha)\Gamma(j-\alpha+1)}{(1-r)\alpha\Gamma(j+2)} & \text{ if } \alpha \in (1,2)\\
      \displaystyle\1_{\{j=1\}} \frac{1}{1-r} & \text{ if } \alpha =2.
    \end{cases}
  \]
  For $j \geq 1$, the other particles at time $r \in [0,1)$ die and are replaced by $j+1$ offspring at rate $q_{j+1}^{(r)}$.	Thus for each $r \in [0,1)$, the number of particles which are at distance $r$ from the root is distributed $\Z(r)$. The process $\Z(r)$ explodes as $r \rightarrow 1$ so we have infinitely many particles at distance one from the root. The space $B(o,1)$ is the set of particles at distance one from the root and $o$ is the immortal particle that is distance one from the root. For each $v,w \in B(o,1)$ there exists two unique paths from the root ending at the points $v,w$ and these paths deviate at distance $h_{v,w}\geq 0$ from the root. The distance between two points is given by $d_\mathbb S(v,w)=1-h_{v,w}$. The multi-type branching process described is the spine decomposition of an inhomogeneous Galton--Watson process for which each particle at time $r \in [0,1)$ dies gives rise to $j+1$ offspring at rate $q_{j+1}^{(r)}$ as introduced in \cite{cr_spine}.


  In the case $\alpha =2$ we are able to strengthen the convergence in Theorem \ref{thm:main} to that of metric measure spaces and explicitly construct the limiting space $(\mathbb S,d_{\mathbb S},o)=(\mathbb S^{(2)},d^{(2)}_{\mathbb S}, o^{(2)})$ (see Figure \ref{fig:construct_S}). To that end construct a measure $\nu$ on the space $(E,\delta)$ as follows. Let $\nu$ be such that the mass it assigns to each closed ball $B(i,t)$ of radius $t>0$ around $i$ is equal to the asymptotic frequency of the block of $\Pi(t)$ containing $i$. This extends uniquely to a measure on the whole space by Carath\'eodory's extension theorem. Our next result shows the tangent cones of the metric space $(E,\delta)$ equipped with the measure $\nu$.

  \begin{thm}\label{thm:mm-converge}
    In the case when $\alpha =2$ in Theorem \ref{thm:main}, there exists a locally finite measure $\mu$ on the space $(\mathbb S, d_\mathbb S)$ such that for all $i \in \N$,
    \[
      (E,\epsilon^{-1}\delta(\cdot,\cdot),4\epsilon^{-1}\nu,i) \rightarrow (\mathbb  S, d_\mathbb  S,\mu, o)
    \]
    weakly as $\epsilon \rightarrow 0$ under the Gromov--Hausdorff--Prokhorov topology.

    The limiting metric measure space $(\mathbb S,d_\mathbb S,\mu, o)$ is independent from $i$ and can be constructed as follows. Let $W= (W(t):t \in \R)$ be a two-sided Brownian motion on $\R$ and let $\mathcal N := \{t \in \R: W(t)=0\}$. For each $x,y \in \mathcal N$ with $x \leq y$, define the pseudo-metric
    \begin{equation}\label{eq:defn_S_intro}
      d_\mathbb  S(x,y):= \sup\{W(t):t \in [x,  y]\}\vee 0.
    \end{equation}
    and $\mathbb S= \mathcal N / \sim$ where $x \sim y$ if and only if $d_\mathbb  S(x,y)=0$ and $ o=0$. The measure $\mu$ is the projection of the local time measure on $\mathcal N$.
  \end{thm}

  We delay the exact definition of the Gromov--Hausdorff--Prokhorov metric to Section \ref{sec:preliminary}.

  \begin{figure}
    \begin{center}
      \includegraphics[scale=0.8]{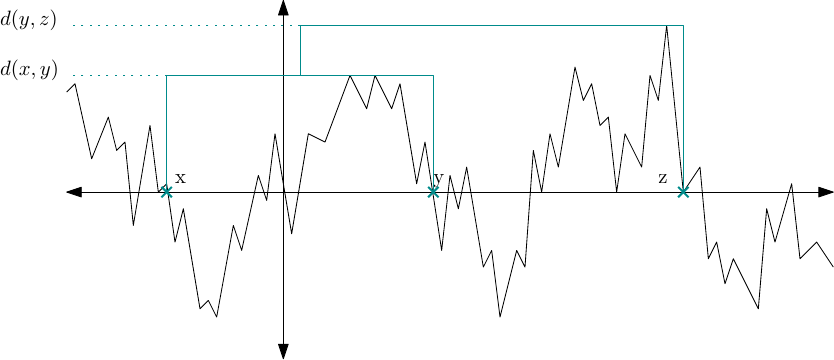}
    \end{center}
    \caption{Construction of the space $(\mathbb  S^{(2)}, d_\mathbb  S^{(2)}, o^{(2)})$ from a two-sided Brownian motion.}\label{fig:construct_S}
  \end{figure}

  Note that Theorem \ref{thm:path_back} in the case $\alpha=2$ can be obtained from Theorem \ref{thm:mm-converge} through some routine computations. The factor of $4$ appearing on the measure in Theorem \ref{thm:mm-converge} is for the following reason. For small $\epsilon>0$ and $t \leq \epsilon$, the number of blocks of a Kingman coalescent at time $t\epsilon$ is roughly $2/(t\epsilon)$. Hence the typical mass of a ball of radius $t$ in $(E,\epsilon^{-1}\delta)$ is roughly $t/2$. On the other hand the typical ball of radius $t$ under $(\mathbb S,d_\mathbb S)$ is the same as the total local-time at level $t$ obtained by a Brownian excursion conditioned to reach level $t$. This has an exponential distribution with expectation $2t$. Thus the two measures are off by roughly a factor of $4$.

  To illustrate the usefulness of the results in the case $\alpha=2$ we present the following corollary. This is an immediate consequence of Theorem \ref{thm:mm-converge}.

  \begin{cor}\label{cor:kingman_asymptotic}
    Let $F(t)$ be the asymptotic frequency of the block containing $1$ in Kingman's coalescent at time $t \geq 0$. Then we have in the sense of weak convergence on the Skorokhod space
    \[
      (\epsilon^{-1}F(\epsilon t): t \geq 0) \rightarrow (X(t):t \geq 0)
    \]
    as $\epsilon \rightarrow 0$.

    The process $X=(X(t): t\geq 0)$ is characterised by the following.
    \begin{enum}
    \item $X(0)=0$ and for $t >0$, $X(t)$ is the sum of two i.i.d. exponential distributions with parameter $1/(2t)$
    \item $X$ is an inhomogeneous compound Poisson process where at time $t>0$ the rate of jumps is given by $2/t$ and the jump distribution is exponential with parameter $1/(2t)$.
    \end{enum}
  \end{cor}
  Note that Corollary \ref{cor:kingman_asymptotic} extends \cite[Corollary 1.3]{natkingbrown} which shows the above convergence for fixed $t \geq 0$.


  \subsection{Directions for further research}
  It would be interesting to see if any of the results presented thus far can be extended to general coalescents. It seems that not every $\Lambda$-coalescent which comes down from infinity will admit a scaling limit (see for example the $\Lambda$-coalescent constructed in \cite[Section 5]{berestycki_small_time}). For coalescents which do not come down from infinity we have that the closed balls are not compact. Taking scaling limits of these coalescents would require a new framework of metric space convergence.

  We note that similar objects appear in the limit when studying the behaviour of fragmentation process with large mass, near the explosion time, see \cite{haas_infinity,gh_behavior1,gh_behavior2}. It would also be interesting to see if there is a relation between coalescents near time zero and fragmentation processes with large mass.

  A simpler question is if Theorem \ref{thm:main} can be extended to use metric-measure space convergence when $\alpha\in(1,2)$? This was done in Theorem \ref{thm:mm-converge} for Kingman's coalescent but the proof of Theorem \ref{thm:mm-converge} relies on a very specific construction of Kingman's coalescent.




  \subsection{Outline of the paper}
  In Section \ref{sec:preliminary} we introduce some background on metric geometry. We assume the reader is familiar with the basic concepts in coalescent theory and excursion theory. We refer to \cite{berestyckibook}, \cite{bertoinfragment} and \cite{revuzyor}. In Section \ref{sec:proof_path_back} we prove Theorem \ref{thm:path_back} for $\alpha \in (1,2)$ and explain the changes needed for the case $\alpha = 2$. Then in Section \ref{sec:proof_a_ranged} we prove Theorem \ref{thm:main}. In Section \ref{sec:proof_a=2} we shall prove Theorem \ref{thm:mm-converge}.

  \subsection*{Acknowledgements}

  I would like to thank my supervisor Nathana\"el Berestycki for introducing me to the problem and his continuous assistance throughout. My thanks to James Norris and Jean Bertoin who kindly pointed out a mistake in a previous version of this paper and the two anonymous referees who helped to improve on the paper substantially.

  \section{Convergence of metric spaces}\label{sec:preliminary}
  In this section we briefly review some basic notions of convergence of metric spaces. For a detailed treatment of the material refer to \cite{metricgeometry}.

  Here we introduce the Gromov--Hausdorff metric used in Theorem \ref{thm:main} and the Gromov--Hausdorff--Prokhorov metric in Theorem \ref{thm:mm-converge}. We start by defining a metric on certain metric spaces without measures, called the Gromov--Hausdorff metric. We introduce this by first defining the Gromov--Hausdorff metric on compact metric spaces. Consider two compact metric spaces $(X,d_X)$ and $(Y,d_Y)$. The compact Gromov--Hausdorff distance $d_{cGH}((X,d_X),(Y,d_Y))$ is constructed as follows. Let $(Z,d_Z)$ be a metric space such that there exists isometric embeddings $\phi_X:(X,d_X) \rightarrow (Z,d_Z)$, $\phi_Y:(Y,d_Y)\rightarrow (Z,d_Z)$, then
  \begin{equation}\label{eq:defn_gh}
    d_{cGH}((X,d_X),(Y,d_Y)):= \inf\{d^Z_H(\phi_X(X),\phi_Y(Y))\}
  \end{equation}
  where the infimum is over all metric spaces $(Z,d_Z)$ with the above property and
  \[
    d^Z_H(A,B):=\inf\{\epsilon>0: B \subset \{z\in Z: \dist_{d_Z}(z,A)<\epsilon\}\text{ and }A \subset \{z\in Z: \dist_{d_Z}(z,B)<\epsilon\}\}
  \]
  is the Hausdorff distance in $(Z,d_Z)$. In particular, $d_{cGH}((X,d_X),(Y,d_Y))=0$ if and only if $(X,d_X)$ and $(Y,d_Y)$ are isometric.

  Now we define the Gromov--Hausdorff--Prokhorov metric on compact spaces. Suppose in addition we have two finite measures $\mu$ and $\nu$ defined on the spaces $(X,d_X)$ and $(Y,d_Y)$ respectively. Again let $(Z,d_Z)$ be a metric space such that there exist isometric embeddings $\phi_X:(X,d_X) \rightarrow (Z,d_Z)$, $\phi_Y:(Y,d_Y)\rightarrow (Z,d_Z)$. Then $\mu^*=\mu\circ \phi_X^{-1}$ and $\nu^*=\nu\circ \phi_Y^{-1}$ are measures on the space $(Z,d_Z)$. The Prokhorov metric on $(Z,d_Z)$ is given by
  \[
    d^Z_{Pr}(\mu^*,\nu^*):=\inf\{\epsilon>0:\mu^*(A)\leq \nu^*(A^\epsilon)+\epsilon \text{ and }\nu^*(A)\leq \mu^*(A^\epsilon)+\epsilon \quad \forall \text{ measurable } A\}
  \]
  where $A^\epsilon:=\{z\in Z: \dist_{d_Z}(z,A)<\epsilon\}$. Then the compact Gromov--Hausdorff--Prokhorov distance is given by
  \[
    d_{cGHP}((X,d_X,\mu), (Y,d_Y,\nu))= \inf\{d^Z_H(\phi_X(X),\phi_Y(Y)) + d^Z_{Pr}(\mu^*,\nu^*)\}
  \]
  where the infimum is over all metric spaces $(Z,d_Z)$ with the above property.

  In this paper we work with non-compact spaces and there are several ways to extend the definition above to a certain class of non-compact metric spaces. We now introduce our notion of the Gromov--Hausdorff and Gromov--Hausdorff--Prokhorov distance on non-compact metric spaces satisfying certain properties. Suppose that $(X,d_X,\mu,p_X)$ and $(Y,d_Y,\nu,p_Y)$ are proper Polish pointed metric spaces, that is they are complete, separable and every closed ball is compact. Then the (pointed) Gromov--Hausdorff distance is given by
  \begin{align*}
    d_{GH}((X,d_X,p_X), (Y,d_Y,p_Y))= \sum_{n \geq 1} 2^{-n} (1 \wedge d_{cGH}((B(p_X,n),d_X),(B(p_Y,n),d_Y)))
  \end{align*}
  where here and throughout $B(p,r)$ denotes the closed ball of radius $r$ around $p$.

  Denote by $(\mathcal{X},d_{GH})$ the space of proper Polish spaces with a distinguished point, up to isometry, equipped with the Gromov--Hausdorff metric. The space $(\mathcal{X},d_{GH})$ is a Polish space (see \cite{evansstflour}). Some of the properties of metric spaces are preserved under $d_{GH}$ convergence. One such example is when the metric is an ultra--metric. A metric $d$ is called an ultra--metric if
  \[
    d(x,y) \leq d(x,y) \vee d(y,z) \qquad \qquad \forall x,y,z.
  \]
  It is not hard to check all the metric spaces in this paper are in fact ultra--metric spaces and that this property is preserved under $d_{GH}$ convergence.

  Suppose that $(X,d_X,\mu,p_X)$ and $(Y,d_Y,\nu,p_Y)$ are proper Polish pointed metric spaces which come equipped with two measures $\mu$ and $\nu$ respectively. Suppose further that both measures are finite on compact sets. Then the (pointed) Gromov--Hausdorff--Prokhorov distance is given by
  \[
    d_{GHP}((X,d_X,\mu,p_X), (Y,d_Y,\nu,p_Y))= \sum_{n \geq 1} 2^{-n} (1 \wedge d_{cGHP}((B(p_X,n),d_X,\mu),(B(p_Y,n),d_Y,\nu))).
  \]
  Denote by $(\mathcal{X}_\mu,d_{GHP})$ the space of proper Polish spaces with a distinguished point, up to measure preserving isometries, equipped with the Gromov--Hausdorff--Prokhorov metric. Then space $(\mathcal{X}_\mu,d_{GHP})$ is a Polish space (see for example \cite{abraham_ghp}).

  It is not hard to check that the space $(E,\delta)$ is compact and the measure $\nu$ is finite. Further it can be seen that the limiting space $(\mathbb S,d_\mathbb S)$ for $\alpha=2$ in Theorem \ref{thm:mm-converge} is a proper Polish space and the measure $\mu$ is finite on compact sets.

  \section{SRV(\texorpdfstring{$\alpha$}{a}) Case}
  \subsection{Proof of Theorem \ref{thm:path_back}}\label{sec:proof_path_back}

  Throughout this proof we omit the superscript $\alpha$ from the notation and assume that $\alpha \in (1,2)$. The proof for the case $\alpha =2$ follows analogously and we remark the only alteration to the proof that is required for $\alpha=2$.

  Let $\Pi=(\Pi(t):t\geq 0)$ denote a $\Lambda$-coalescent such that $\Lambda(dp)=f(p)\, dp$ with $f(p)\sim A_\Lambda p^{1-\alpha}$ as $p \rightarrow 0$. Let $N=(N(t):t>0)$ denote the number of blocks of process $\Pi$. Recall that $\Z_\epsilon(r)$ is the number of blocks at time $(1-r)\epsilon$ that make up the block containing $1$ at time $\epsilon$. We will show the convergence result by showing that $\Z_\epsilon(r)$ almost satisfies a certain martingale problem for small $\epsilon>0$.

  We will simplify the notation further by writing $\Pi^\epsilon(r):=\Pi((1-r)\epsilon)$ and $N^\epsilon(r):=N((1-r)\epsilon)$ for $r \in [0,1)$. We will denote by $\F^\epsilon_r$ the natural filtration of $\Z_\epsilon(r)$. Before we begin the proof let us describe an outline.

  \subsubsection*{Outline of the proof}
    Our strategy in proving Theorem \ref{thm:path_back}, the convergence $\Z_\epsilon \rightarrow \Z$ as $\epsilon \rightarrow 0$, is separated into two steps: (1) showing that $\{\Z_\epsilon\}_{\epsilon>0}$ is tight (Lemma \ref{lemma:tightness}) and (2) showing that every subsequence of $\Z_\epsilon$ which converges satisfies the martingale problem for the generator $L_r$ of the limit $\Z$, which given in Theorem \ref{thm:path_back} (Lemma \ref{lemma:martingale_close}).
    The latter statement turns out to be somewhat harder to show.

    The natural approach to show (2) would be to identify $\lim_{\delta\downarrow 0} \delta^{-1}\P(\Z_\epsilon(r+\delta)-\Z_\epsilon(r)=j|\F^\epsilon_r)$ as this corresponds to the $Q$-matrix of the process. We could then take the limit as $\epsilon \rightarrow 0$ and conclude the proof by using results from \cite{MR838085}. Unfortunately we are unable to show the existence of $\lim_{\delta\downarrow 0} \delta^{-1}\P(\Z_\epsilon(r+\delta)-\Z_\epsilon(r)=j|\F^\epsilon_r)$, so instead we rely on estimating $\lim_{\delta\downarrow 0}\delta^{-1}\P(\Z_\epsilon(r+\delta)-\Z_\epsilon(r)=j;A)$ uniformly over $A \in \F^\epsilon_r$ and $r$ inside a compact interval. For fixed $\epsilon>0$ our upper and lower bound will differ however they become arbitrarily close to each other as $\epsilon \downarrow 0$. This will be sufficient to show that the limit $\Z$ satisfies the martingale problem.

    Our strategy for obtaining bounds on $\delta^{-1}\P(\Z_\epsilon(r+\delta)-\Z_\epsilon(r)=j;A)$ is by using Bayes' formula to write this probability in terms of events that move forwards in time for the coalescent process. We do this as follows. The event $\{\Z_\epsilon(r+\delta)-\Z_\epsilon(r)=j\}$ means that there was a coalescent event during $((1-r-\delta)\epsilon,(1-r)\epsilon)$ merging $j+1$ blocks together, afterwhich the resulting block as well as $\Z_\epsilon(r)-1$ other blocks merge during the interval $((1-r)\epsilon,\epsilon)$ to form the block containing $1$ at time $\epsilon$. Calculating the probability of the first merger is simple. Next, we apply the Markov property after the first merger (at time $(1-r)\epsilon$) and ask for the conditional probability the event $A$ holds as well as the resulting block together with $\Z_\epsilon(r)-1$ other blocks merging during the interval $((1-r)\epsilon,\epsilon)$ to form the block containing $1$ at time $\epsilon$.
    It turns out that the unconditional probability of this event is much easier to calculate. Thus firstly in Lemma \ref{lemma:beta_markov} we estimate the difference between the conditional and the unconditional probability of such events.
    Then in Lemma \ref{lemma:mu_close_to_Y} we use the strategy described here to derive bounds on $\lim_{\delta\downarrow 0}\delta^{-1}\P(\Z_\epsilon(r+\delta)-\Z_\epsilon(r)=j;A)$. Finally in Lemma \ref{lemma:limit_Y} we show the convergence of the bounds obtained in Lemma \ref{lemma:mu_close_to_Y} as $\epsilon \downarrow 0$.

    \bigskip

  Recall that the event $\{\Z_\epsilon(r)=\ell\}$ means that the block containing $1$ at time $\epsilon$ is made up of exactly $\ell$ blocks at time $(1-r)\epsilon$. It will be very convenient to fix the indices of these blocks. For $\ell \leq n$, define
  \begin{align*}
    \mathcal{S}_{\ell,n}&:=\{ \mathbf{z}=(z_j)_{j=1}^\ell: 1= z_1<\dots<z_\ell\leq n\} 
  \end{align*}
  which will represent the set indices of the blocks at time $(1-r)\epsilon$ which eventually merge and form the block containing $1$ at time $\epsilon$. For $\mathbf{z}\in \mathcal{S}_{\ell,n}$ and $r \in [0,1)$ consider the event
  \begin{equation}\label{eq:beta_kappa_def}
    \kappa(r,\mathbf{z}):=\left\{\Pi^\epsilon_1(0)=\bigcup_{i \in \mathbf{z}}\Pi^\epsilon_i(r)\right\} \cap \{\Pi^\epsilon_i(r) \neq \emptyset, \,\forall i \in \mathbf z\}.
  \end{equation}
  where $\Pi^\epsilon_i(r)$ denotes the $i$--th block of $\Pi^\epsilon(r)$ where the blocks are ordered by infimum.
  In words, this is the event that the block containing $1$ at time $\epsilon$ is made up of blocks with labels given by $\mathbf z$ at time $(1-r)\epsilon$.
  For example if $\Pi^\epsilon(0)=\{\{1,3,5\}, \{2,4\}\}$ and $\Pi^\epsilon(r)=\{\{1,3\},\{2\},\{4\},\{5\}\}$, then the event $\kappa(r,\{1,4\})$ would hold since $\Pi^\epsilon_1(0)=\{1,3,5\}=\{1,3\}\cup\{5\}=\Pi^\epsilon_1(r) \cup \Pi^\epsilon_4(r)$.
  Thus for fixed $r \in [0,1)$, $\kappa(r,\mathbf z) \cap \kappa(r,\mathbf z') = \emptyset$ for $\mathbf z \neq \mathbf z'$, further
  \begin{equation}\label{eq:Z_kappa}
    \{\Z_\epsilon(r)=\ell\} \cap \{N^\epsilon(r)=n\}= \bigcup_{\mathbf{z}\in \mathcal{S}_{\ell,n}} \kappa(r,\mathbf{z})\cap \{N^\epsilon(r)=n\}.
  \end{equation}

  For the next lemma let $R_n$ be the map which maps a partition on $\N$ to a partition on $[n]$ by projection. The next lemma estimates the effects of a merger of $j+1$ blocks during $((1-r-\delta)\epsilon,(1-r)\epsilon)$ on the probability of the event $A$.
  \begin{lemma}\label{lemma:beta_markov}
    For any $\epsilon>0$, $r \in [0,1)$, $A \in \F^\epsilon_r$ and $j<n$, it holds that
    \begin{equation}\label{eq:beta_markov_ineq}
      |\P(A|N^\epsilon(r)=n)-\P(A|N^\epsilon(r)=n-j)| \leq j(1-e^{-\epsilon}).
    \end{equation}
  \end{lemma}
  Let us give heuristics about the proof before we begin. The event $A$ asks questions concerning the particles which merge with block containing $1$ during the time period $[(1-r)\epsilon,\epsilon]$.
  Now we have two configurations for $N^\epsilon(r)=N((1-r)\epsilon)$, one of which has $j$ extra particles. On the event that these extra $j$ particles do not interact with the block containing $1$ during the time interval $[(1-r)\epsilon,\epsilon]$, the event $A$ has the same probability in each of these two configurations.
  The probability on the right hand side of \eqref{eq:beta_markov_ineq} is an estimate for the event that the extra $j$ particles do interact with the block containing $1$ during the time interval $[(1-r)\epsilon,\epsilon]$.
  \begin{proof}
    Fix $r \in [0,1)$, $n \in \N$ and $j<n$ throughout. In order to make the sketch above rigorous we will use a monotone class argument. Let $\mathcal M$ be the set of $A \in \F^\epsilon_r$ such that \eqref{eq:beta_markov_ineq} holds. Suppose that $A_1,A_2,\dots \in \mathcal M$ such that $A_1 \subset A_2 \subset \dots$ and let $A= \bigcup_{i \geq 1} A_i$. Then
    \[
      |\P(A|N^\epsilon(r)=n)-\P(A|N^\epsilon(r)=n-j)| \leq \limsup_{i \rightarrow \infty} |\P(A_i|N^\epsilon(r)=n)-\P(A_i|N^\epsilon(r)=n-j)| \leq j(1-e^{-\epsilon}).
    \]
    and hence $\mathcal M$ is a monotone class.
    All that remains to be shown is that $\mathcal M$ contains a $\pi$-system which generates the $\sigma$-algebra $\mathcal F^\epsilon_r$ since once this is shown the result follows from monotone class theorem.

    Condition on $N^\epsilon(r)=n$, so that $\Pi^\epsilon(r)=(\Pi_1^\epsilon(r),\dots,\Pi_n^\epsilon(r))$.
    Let $\mathcal{A}$ denote the set of events of the form
    \begin{align*}
      \bigcup_{\mathbf{z} \in \mathcal{S}_{\ell,n}}\kappa(u,\mathbf{z}),\qquad & u \leq r, \ell \leq n.
    \end{align*}
    Let $\mathcal{A}'$ denote the $\pi$-system generated by $\mathcal{A}$. We will show that $\mathcal A'\subset \mathcal M$ by showing that every $A \in \mathcal A'$ satisfies \eqref{eq:beta_markov_ineq}.

    Note that conditioning on $N^\epsilon(r)=n$ implies that for each $\ell \leq n $ and $u \leq r$, $\{\Z_\epsilon(u)=\ell\}=\bigcup_{\mathbf{z} \in \mathcal{S}_{\ell,n}}\kappa(u,\mathbf{z})$. Hence $\mathcal A'$ generates the $\sigma$-algebra $\F^\epsilon_r$.
    Henceforth let $A \in \mathcal{A}'$ be fixed and denote by $A_1,\dots,A_m\in \mathcal{A}$ the elements such that $A = A_1\cap \dots \cap A_m$.

    For any $u \leq r$, $\ell \leq n$, applying the Markov property we have
    \begin{align}\label{eq:bm_project}
      &\P\left(\bigcup_{\mathbf{z} \in \mathcal{S}_{\ell,n}}\kappa(u,\mathbf{z}) \middle |   N^\epsilon(r)=n\right) \nonumber\\
      & \quad =  \P\left(\bigcup_{\mathbf{z} \in \mathcal{S}_{\ell,n}}\{R_n\Pi_1(r\epsilon)=R_n\Pi_{z_1}((r-u)\epsilon)\cup \dots \cup R_n\Pi_{z_\ell}((r-u)\epsilon)\}\right)
    \end{align}
    where $R_n \Pi_i(t)$ is the $i$--th block of $R_n \Pi$ at time $t>0$. For each $A_i$, $i=1,\dots,m$, we denote by $\tilde A_i^n$ the event on the right hand side of \eqref{eq:bm_project} so that $\tilde A^n_i \in \sigma(R_n\Pi(s):s \leq r\epsilon)$ and
    \begin{equation}\label{eq:A_to_tilde_A}
      \P(A_i|N^\epsilon(r)=n)=\P(\tilde A^n_i) \qquad i\leq m
    \end{equation}
    We will show that $\P(\cap_{i\leq m}\tilde A^n_i \Delta \cap_{i\leq m} \tilde A^{n-j}_i)\leq j(1-e^{-\epsilon})$, which will show \eqref{eq:beta_markov_ineq} for all $A \in \mathcal A'$.

    For any $u \leq r$, $\ell \leq n$ and $\mathbf{z} \in \mathcal{S}_{\ell,n}$,
    \begin{align}\label{eq:bm_mess}
      &\{R_n\Pi_1(r\epsilon)=R_n\Pi_{z_1}((r-u)\epsilon)\cup \dots \cup R_n\Pi_{z_\ell}((r-u)\epsilon)\}\cap \bigcap_{k=n-j+1}^n \{k \notin R_n\Pi_1(r\epsilon)\}\\
      & \quad =\{R_{n-j}\Pi_1(r\epsilon)=R_{n-j}\Pi_{z_1}((r-u)\epsilon)\cup \dots \cup R_{n-j}\Pi_{z_\ell}((r-u)\epsilon)\}\cap \bigcap_{k=n-j+1}^n \{k \notin R_n\Pi_1(r\epsilon)\}.\nonumber
    \end{align}
    Indeed, consider the blocks at time $(r-u)\epsilon$ that coalesce by time $r\epsilon$ form the block containing $1$. The only way these can differ between the restrictions to $\{1,\dots,n-j\}$ and to $\{1,\dots,n\}$ is if at least one of $\{n-j+1,\dots,n\}$ coalesced with $1$ by time $r\epsilon$.

    Recall that for each $i \in m$, $\tilde A_i^n$ and $\tilde A_i^{n-j}$ are given by \eqref{eq:A_to_tilde_A}. Similar to \eqref{eq:bm_mess},
    \[
      \bigcap_{i=1}^m \tilde A_i^n \cap \bigcap_{k=n-j+1}^n \{k \notin R_n\Pi_1(r\epsilon)\} = \bigcap_{i=1}^m \tilde A_i^{n-j} \cap \bigcap_{k=n-j+1}^n \{k \notin R_n\Pi_1(r\epsilon)\}.
    \]

    The event $\bigcap_{k=n-j+1}^n \{k \notin R_n\Pi_1(r\epsilon)\}$ does not depend on $u$, it suffices to show that this event has probability at least $1-j(1-e^{-\epsilon})$ as
    \[
      \left(\bigcap_{i=1}^m \tilde A_i^n  \right)\Delta \left(\bigcap_{i=1}^m \tilde A_i^{n-j}\right) \subset \left(\bigcap_{k=n-j+1}^n \{k \notin R_n\Pi_1(r\epsilon)\}\right)^c.
    \]

    Now notice that $\P(\{k \in R_n\Pi_1(r\epsilon)\})= 1-e^{-r\epsilon}\leq1- e^{-\epsilon}$ and hence
    \[
      \P\left(\bigcup_{k=n-j+1}^n \{k \in R_n\Pi_1(r\epsilon)\}\right)\leq j(1-e^{-\epsilon}).
    \]
    Thus it follows that \eqref{eq:beta_markov_ineq} holds for every $A \in \mathcal{A}'$.
  \end{proof}

  Let $\lambda_{n,k}$ be the rate at which a collision involving exactly $k$ fixed blocks occurs when there are currently $n$ blocks present. Define
  \[\label{eq:defn_gamma_lambda}
    \gamma_{n,k}:=\binom{n}{k} \lambda_{n,k}
  \]
  which is the total rate of mergers of $k$ blocks when $n$ blocks are present. Let us fix $M>1$ and define the process $\Z_\epsilon^M$ by
  \[
    \Z^M_\epsilon(r):=\Z_\epsilon(r) \wedge M \qquad r \in [0,1).
  \]

  The next lemma gives is the central lemma of the proof and gives estimates on the jump probabilities of $\Z_\epsilon$. The expectation of the random variable which these probabilities are close to will be identified in Lemma \ref{lemma:limit_Y}.

  \begin{lemma}\label{lemma:mu_close_to_Y}
    Let $K \subset [0,1)$ be a compact set, then for any $j \leq M$,
    \begin{align*}
      \sup_{r\in K}\sup_{A \in \F^\epsilon_r}\limsup_{\delta\rightarrow 0}&\left|\frac{1}{\delta}\P(A;\Z^M_\epsilon(r+\delta)-\Z^M_\epsilon(r)=j)\right.- \left.\E\left[\epsilon(j+\Z^M_\epsilon(r))\frac{\gamma_{N^\epsilon(r),j+1}}{N^\epsilon(r)}\1_{A; \Z^M_\epsilon(r)<M-j}\right] \right| \\
      &\qquad\qquad \leq j(1-e^{-\epsilon})\frac{3M(M+1)}{2}\E\left[\sup_{r\in K} \epsilon\frac{\gamma_{N^\epsilon(r),j+1}}{N^\epsilon(r)}\right].
    \end{align*}
  \end{lemma}

  \begin{figure}
    \begin{center}
      \includegraphics{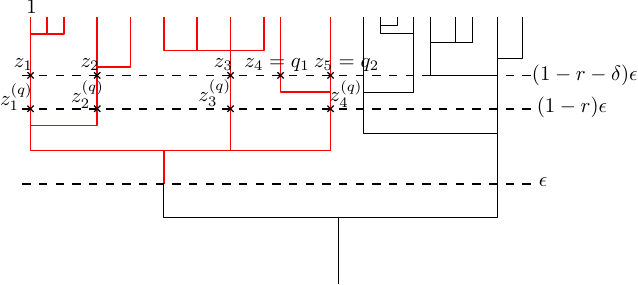}
    \end{center}
    \caption{Example showing how to evaluate interpret $\P(\Z_\epsilon(r+\delta)=\ell;\Z_\epsilon(r)=\ell-j | N^\epsilon(r+\delta)=n)$. In this figure $\mathbf z=\{1,2,3,4,5\}$, $\mathbf q=\{4,5\}$ and $\mathbf z^{(\mathbf q)}=\{1,2,3,4\}$}\label{fig:sketch}
  \end{figure}

  Before we begin let us give a small outline of the proof. Suppose for simplicity that we want to work out $\P(\Z_\epsilon(r+\delta)=\ell;\Z_\epsilon(r)=\ell-j | N^\epsilon(r+\delta)=n)$. This means that there are $\ell$ blocks at time $(1-r-\delta)\epsilon$ with labels $\mathbf z \in \mathcal S_{\ell,n}$ which form the block containing $1$ at time $\epsilon$. Furthermore during the time interval $[(1-r-\delta)\epsilon, (1-r)\epsilon]$ there must have been a merger involving some labels $\mathbf q \subset \mathbf z$ with $|\mathbf q|=j+1$, see Figure \ref{fig:sketch}.

  The key to unravelling this statement is to successively use Bayes' formula and the Markov property of the coalescent. First we can work out the probability that there is merger which merges the blocks with labels $\mathbf q$ during the time interval $[(1-r-\delta)\epsilon, (1-r)\epsilon]$. Conditionally on this, we work out the probability that indices $\mathbf z^{(\mathbf q)}$ of size $\ell-j$ at time $(1-r)\epsilon$ must make up the block containing $1$ at time $\epsilon$. The indices $\mathbf z^{(\mathbf q)}$ are the indices which correspond to the positions of indices $\mathbf z$ after the merger involving blocks with labels $\mathbf q$ has taken place (see again Figure \ref{fig:sketch}).

  \begin{proof}

    Let $n \in \N$, $\ell<n\wedge M$, $j \leq (n-\ell)\wedge M$, $r \in [0,1)$ and $A \in \mathcal{F}^\epsilon_r$. Suppose further that $\pi \in \mathcal{P}_\infty$ with $n$ blocks and that $\delta>0$ is small. Let $\tilde \P$ be the measure $\P$ conditioned on having at most one coalescent event during the time interval $((1-r-\delta)\epsilon,(1-r)\epsilon)$. Then it suffices to show the result for $\tilde \P$ since for any event $A'$, $\P(A')=\tilde \P(A')+o(\delta)$ uniformly over $r \in K$ a compact set $K$.

    For $\mathbf{q}\subset \mathbf z$ let $M_\delta(\mathbf{q})$ denote the event that there is a coalescent event in the interval $((1-r-\delta)\epsilon,(1-r)\epsilon)$ which merges the blocks of $\Pi^\epsilon(r+\delta)$ with labels $\mathbf{q}$. We claim that that for any $\mathbf{z}\in \mathcal{S}_{\ell+j,n}$
    \begin{align}\label{eq:Z_to_merge}
      &\tilde\P(A;\kappa(r+\delta,\mathbf{z}); \Z_\epsilon(r)=\ell ; \Pi^\epsilon(r+\delta)=\pi)\nonumber\\
      &\qquad= \sum_{\substack{\mathbf{q}\subset\mathbf{z} \\ |\mathbf q|=j+1}} \tilde\P(A; M_\delta(\mathbf{q}); \kappa(r,\mathbf{z}^{(\mathbf{q})}); \Pi^\epsilon(r+\delta)=\pi)
    \end{align}
    where $\mathbf{z}^{(\mathbf{q})}\in \mathcal{S}_{\ell,n-j}$ represents the position of the indices $\mathbf{z}$ after the merger involving indices $\mathbf{q}$ has occurred (see Figure \ref{fig:sketch}). Indeed, suppose there is only one coalescent event during the interval $((1-r-\delta)\epsilon,(1-r)\epsilon)$. On the event $\kappa(r+\delta,\mathbf{z})\cap \{\Pi^\epsilon(r+\delta)=\pi\}$, $\Z_\epsilon(r)=\ell$ if and only if during this coalescent event exactly $j+1$ blocks $\pi_{q_1}, \dots, \pi_{q_{j+1}}$ merge with $q_i\in \mathbf{z}$ for each $i \leq j+1$. We set $\mathbf{q}=\{q_1,\dots,q_{j+1}\}$.
    After this merger the blocks of $\pi$ with labels given by $\mathbf{z}$ now have new labels $\mathbf{z}^{(\mathbf{q})}\in \mathcal{S}_{\ell,n-j}$. Then we require that the blocks of $\Pi^\epsilon(r)$ with labels $\mathbf{z}^{(\mathbf{q})}$ eventually merge to give $\Pi^\epsilon_1(1)$, i.e. $\kappa(r,\mathbf{z}^{(\mathbf{q})})$ holds.

    Similar to $\mathbf z^{(\mathbf q)}$, let $\pi^{(\mathbf{q})}\in \mathcal{P}_\infty$ be the partition obtained from $\pi$ by merging the blocks with labels $\mathbf{q}$.
    Markov property of the coalescent implies that
    \[
      \tilde\P(A;\kappa(r,\mathbf{z}^{(\mathbf{q})})| \Pi^\epsilon(r+\delta)=\pi; M_\delta(\mathbf{q})) = \tilde\P(A;\kappa(r,\mathbf{z}^{(\mathbf{q})})| \Pi^\epsilon(r)=\pi^{(\mathbf{q})}).
    \]
    The above equation and \eqref{eq:Z_to_merge} gives
    \begin{align}\label{eq:rate_computation}
      &\tilde\P(A;\kappa(r+\delta,\mathbf{z}); \Z_\epsilon(r)=\ell ; \Pi^\epsilon(r+\delta)=\pi) \nonumber\\
      &\quad= \sum_{\substack{\mathbf{q}\subset\mathbf{z} \\ |\mathbf q|=j+1}}  \tilde\P( M_\delta(\mathbf{q})|\Pi^\epsilon(r+\delta)=\pi) \tilde\P(A;\kappa(r,\mathbf{z}^{(\mathbf{q})})| \Pi^\epsilon(r+\delta)=\pi; M_\delta(\mathbf{q}))\tilde\P(\Pi^\epsilon(r+\delta)=\pi)  \nonumber\\
      &\quad =\sum_{\substack{\mathbf{q}\subset\mathbf{z} \\ |\mathbf q|=j+1}}  \tilde\P( M_\delta(\mathbf{q})|\Pi^\epsilon(r+\delta)=\pi) \tilde\P(A;\kappa(r,\mathbf{z}^{(\mathbf{q})})| \Pi^\epsilon(r)=\pi^{(\mathbf{q})})\tilde\P(\Pi^\epsilon(r+\delta)=\pi) .
    \end{align}

    Now we compute each term inside the sum in \eqref{eq:rate_computation}. Using the exchangeability of the coalescent for the first term gives that
    \begin{equation}\label{eq:term_rate}
      \tilde\P(M_\delta(\mathbf{q})|\Pi^\epsilon(r+\delta)=\pi) = \tilde\P(M_\delta(\{1,\dots,j+1\})|N^\epsilon(r+\delta)=n).
    \end{equation}
    For the second term notice first that by exchangeability for any $\mathbf z' \in \mathcal S_{\ell,n-j}$ and any partition $\pi'$ with $n-j$ blocks,
    \[
    \tilde\P(A;\kappa(r,\mathbf{z}^{(\mathbf{q})})| \Pi^\epsilon(r)=\pi^{(\mathbf{q})}) = \tilde\P(A;\kappa(r,\mathbf{z}')| \Pi^\epsilon(r)=\pi').
    \]
    Thus it follows that
    \begin{align*}
      \P(A;Z_\epsilon(r)=\ell;N^\epsilon(r)=n-j) &= \sum_{\#\pi=n-j} \sum_{\mathbf z' \in \mathcal S_{\ell,n-j}}\tilde\P(A;\kappa(r,\mathbf{z}')| \Pi^\epsilon(r)=\pi') \tilde\P(\Pi^\epsilon(r)=\pi')\\
      &=|\mathcal S_{\ell,n-j}|\tilde\P(A;\kappa(r,\mathbf{z}^{(\mathbf{q})})| \Pi^\epsilon(r)=\pi^{(\mathbf{q})}) \tilde\P(N^\epsilon(r)=n-j).
    \end{align*}
    Now $|\mathcal S_{\ell,n-j}|=\binom{n-j-1}{\ell-1}$, thus
    \begin{equation}\label{eq:term_second}
      \tilde\P(A; \kappa(r,\mathbf{z}^{(\mathbf{q})})|\Pi^\epsilon(r)=\pi^{(\mathbf{q})}) = \tilde\P(A;Z_\epsilon(r)=\ell|N^\epsilon(r)=n-j) \binom{n-j-1}{\ell-1}^{-1}.
    \end{equation}
    Plugging \eqref{eq:term_rate} and \eqref{eq:term_second} into \eqref{eq:rate_computation} and summing gives
    %
    \begin{align}
      &\tilde\P(A; Z_\epsilon(r+\delta)=\ell+j ; \Z_\epsilon(r)=\ell ; N^\epsilon(r+\delta)=n)\nonumber\\
      &\quad =\sum_{\#\pi=n} \sum_{\mathbf z \in \mathcal S_{\ell+j,n}}\tilde\P(A;\kappa(r+\delta,\mathbf{z}); \Z_\epsilon(r)=\ell ; \Pi^\epsilon(r+\delta)=\pi)\nonumber\\
      &\quad = \sum_{\#\pi=n} \sum_{\mathbf z \in \mathcal S_{\ell+j,n}}\sum_{\substack{\mathbf{q}\subset\mathbf{z} \\ |\mathbf q|=j+1}}  \tilde\P( M_\delta(\mathbf{q})|\Pi^\epsilon(r+\delta)=\pi) \tilde\P(A;\kappa(r,\mathbf{z}^{(\mathbf{q})})| \Pi^\epsilon(r)=\pi^{(\mathbf{q})})\tilde\P(\Pi^\epsilon(r+\delta)=\pi)\nonumber\\
      &\quad =  \sum_{\#\pi=n} \sum_{\mathbf z \in \mathcal S_{\ell+j,n}}\sum_{\substack{\mathbf{q}\subset\mathbf{z} \\ |\mathbf q|=j+1}}  \frac{\tilde\P(M_\delta(\{1,\dots,j\})|N^\epsilon(r+\delta)=n) \tilde\P(A;Z_\epsilon(r)=\ell|N^\epsilon(r)=n-j) \tilde\P(\Pi^\epsilon(r+\delta)=\pi)}{\binom{n-j-1}{\ell-1}}\nonumber\\
      &\quad = \frac{\binom{n-1}{\ell+j-1}\binom{\ell+j}{j}}{\binom{n-j-1}{\ell-1}}\tilde\P(M_\delta(\{1,\dots,j+1\})|N^\epsilon(r+\delta)=n) \tilde\P(A;Z_\epsilon(r)=\ell|N^\epsilon(r)=n-j) \tilde\P(N^\epsilon(r+\delta)=n) \nonumber \\
      &\quad = \frac{j+\ell}{n} \binom{n}{j+1} \tilde\P(M_\delta(\{1,\dots,j+1\})|N^\epsilon(r+\delta)=n) \tilde\P(A;Z_\epsilon(r)=\ell|N^\epsilon(r)=n-j) \tilde\P(N^\epsilon(r+\delta)=n)  \nonumber
    \end{align}
    where $\#\pi$ denotes the number of blocks of $\pi$. Diving by $\delta$ and taking limits gives
    \begin{align}\label{eq:sum_me_with_n}
      \lim_{\delta\downarrow 0}\frac{1}{\delta} &\tilde\P(A; Z_\epsilon(r+\delta)=\ell+j ; \Z_\epsilon(r)=\ell ; N^\epsilon(r+\delta)=n) \\
      &= \frac{j+\ell}{n} \epsilon\gamma_{n,j+1}\tilde\P(A;Z_\epsilon(r)=\ell|N^\epsilon(r)=n-j) \tilde\P(N^\epsilon(r)=n)\nonumber
    \end{align}
    where we have used the fact that
    \[
    \binom{n}{j+1}\lim_{\delta\downarrow 0}\frac{1}{\delta} \tilde\P(M_\delta(\{1,\dots,j+1\})|N^\epsilon(r+\delta)=n) = \epsilon\gamma_{n,j+1}.
    \]

    We wish to sum \eqref{eq:sum_me_with_n} over $n$ and then apply dominated convergence to swap the limit with the summation. For $\delta'< 1-r$ there exists a constant $C>0$ such that
    \begin{align*}
      \sup_{\delta < \delta'}\frac{1}{\delta}\tilde\P(A; &Z_\epsilon(r+\delta)=\ell+j ; \Z_\epsilon(r)=\ell ; N^\epsilon(r+\delta)=n) \\
      &\leq C\sup_{\delta < \delta'}\frac{1}{\delta}\tilde\P(A; Z_\epsilon(r+\delta)=\ell+j ; \Z_\epsilon(r)=\ell ; N^\epsilon(r)=n)\\
      &\leq C\sup_{\delta < \delta'}\frac{1}{\delta}C\P(\text{ there is a coalescent event during } ((1-r-\delta)\epsilon,(1-r)\epsilon) ; N^\epsilon(r)=n)\\
      &= C\sup_{\delta < \delta'}\frac{1}{\delta}(1-e^{-\delta\epsilon \lambda_n})\tilde\P(N^{\epsilon}(r)=n) \\
      & \leq C\epsilon\lambda_n\P(N^{\epsilon}(r)=n)
    \end{align*}
    where $\lambda_n=\sum_{k \leq n}\binom{n}{k}\lambda_{n,k}$ is the total rate of coalescent events and in the final inequality we have used the fact that for $x \geq 0$, $1-e^{-x}\leq x$. The final term on the line above is summable hence by \eqref{eq:sum_me_with_n} and dominated convergence theorem,
    \begin{align}\label{eq:dominated}
      \lim_{\delta\downarrow 0}\frac{1}{\delta}&\tilde\P(A; Z_\epsilon(r+\delta)=\ell+j ; \Z_\epsilon(r)=\ell) \\
      &= \sum_{n=\ell+j}^\infty\frac{j+\ell}{n} \epsilon\gamma_{n,j+1}\tilde\P(A;Z_\epsilon(r)=\ell|N^\epsilon(r)=n-j) \tilde\P(N^\epsilon(r)=n)\nonumber.
    \end{align}

    Next Lemma \ref{lemma:beta_markov} gives
    \begin{equation}\label{eq:jump_estimate}
      |\tilde\P(A;Z_\epsilon(r)=\ell|N^\epsilon(r)=n-j)-\tilde\P(A;Z_\epsilon(r)=\ell|N^\epsilon(r)=n)|\leq j(1-e^{-\epsilon}).
    \end{equation}
    Using \eqref{eq:jump_estimate} together with \eqref{eq:dominated} we get that
    \begin{align*}
      &\limsup_{\delta \downarrow 0}\left|\frac{1}{\delta}\tilde\P(A;\Z^M_\epsilon(r+\delta)-\Z^M_\epsilon(r)=j)\right. \left. - \sum_{\ell=1}^M\sum_{n = \ell+j}^\infty\epsilon\gamma_{n,j+1}\frac{j+\ell}{n} \tilde\P(A;\Z_\epsilon(r)=\ell; N^\epsilon(r)=n) \right|\\
      &\quad \leq j(1-e^{-\epsilon})\sum_{\ell=1}^M\sum_{n = \ell+j}^\infty\epsilon\gamma_{n,j+1}\frac{j+\ell}{n} \tilde\P(N^\epsilon(r)=n)  \\
      &\quad\leq j(1-e^{-\epsilon})\frac{3M(M+1)}{2}\sum_{n = j+1}^\infty\epsilon\frac{\gamma_{n,j+1}}{n} \tilde\P(N^\epsilon(r)=n)
    \end{align*}
    and the result follows.
  \end{proof}

  Now we can identify the limiting behaviour of the expectations in Lemma \ref{lemma:mu_close_to_Y} when $\epsilon \downarrow 0$.

  \begin{lemma}\label{lemma:limit_Y}
    For any $M>1$, $j<M$ and $K \subset [0,1)$ compact,
    \begin{equation}\label{eq:rates_limit}
      \lim_{\epsilon \rightarrow 0}\E\left[\sup_{r\in K}\left|\epsilon\frac{\gamma_{N^\epsilon(r),j+1}}{N^\epsilon(r)}-A_\Lambda \frac{\Gamma(2-\alpha)\Gamma(j-\alpha+1)}{(1-r)\alpha\Gamma(j+2)}\right|\right]= 0
    \end{equation}
  \end{lemma}

  Let us present the idea of the proof of this lemma. Firstly \cite[Theorem 2]{comingdownlambda} implies that $\epsilon/N^\epsilon(r)\approx (1-r)^{-1} N^{\epsilon}(r)^{-\alpha}$ thus we are effectively seeking to show that
  \[
  \lim_{n \rightarrow \infty}\frac{\gamma_{n,j+1}}{n^\alpha} = A_\Lambda \frac{\Gamma(2-\alpha)\Gamma(j-\alpha+1)}{\alpha\Gamma(j+2)}.
  \]
  This limit is not hard to take and has appeared in literature before in \cite[Lemma 4]{bertoinflows3} and \cite[10]{berestyckismall}. The proof presented below is somewhat more delicate since we would like to obtain a strong notion of convergence.

  \begin{proof}
    Fix $M>1$, $j<M$ and $K \subset [0,1)$ compact. Firstly from \cite[Theorem 2]{comingdownlambda} we have that there exists a constant $C_\alpha>0$ such that
    \begin{equation}\label{eq:beres_speed}
      \lim_{\epsilon\rightarrow 0}\E\left[\sup_{r\in K}\left|\frac{(1-r)\epsilon}{C_\alpha N^\epsilon(r)^{1-\alpha}} -1 \right|^2\right]=0.
    \end{equation}
    One can show (see \cite[XIII.6]{feller}) that this constant is given by
    \[
      C_\alpha=\frac{\alpha}{\Gamma(2-\alpha)}.
    \]
    Using the Cauchy--Schwartz inequality we have
    \begin{align}\label{eq:speed_estimate}
      &\E\left[\sup_{r\in K}\left|\epsilon\frac{\gamma_{N^\epsilon(r),j+1}}{N^\epsilon(r)}-\frac{\alpha}{(1-r)\Gamma(2-\alpha)}\frac{\gamma_{N^\epsilon(r),j+1}}{N^\epsilon(r)^{\alpha}}\right|\right]^2\nonumber=\E\left[\sup_{r\in K}\left|\frac{\gamma_{N^\epsilon(r),j+1}}{N^\epsilon(r)^{\alpha}}\left(\frac{(1-r)\epsilon\Gamma(2-\alpha)}{\alpha N^\epsilon(r)^{1-\alpha}} -1\right) \right|\right]^2\nonumber\\
      &\qquad \leq \E\left[\sup_{r\in K}\left|\frac{\gamma_{N^\epsilon(r),j+1}}{N^\epsilon(r)^{\alpha}}\right|^2\right]\E\left[\sup_{r \in K}\left|  \frac{(1-r)\epsilon\Gamma(2-\alpha)}{\alpha N^\epsilon(r)^{1-\alpha}} -1 \right|^2\right].
    \end{align}
    The second term on the last line converges to zero by \eqref{eq:beres_speed}, thus we focus on the first term in the last line.

    Recall that for $n \in \N$ we have that
    \begin{equation}\label{eq:gamma_remind}
      \gamma_{n,j+1}=\binom{n}{j+1}\lambda_{n,j+1}=\frac{\Gamma(n+1)}{\Gamma(j+2)\Gamma(n-j)} \int_0^1 p^{j-1}(1-p)^{n-j-1}\,\Lambda(dp).
    \end{equation}
    Moreover $\Lambda(dp)=f(p)\,dp$ where $f(p) \sim A_\Lambda p^{1-\alpha}$ as $p \rightarrow 0$. Fix $\eta_0>0$, then there exist a $p_0\in (0,1)$ such that whenever $p<p_0$ we have $|f(p)-A_\Lambda p^{1-\alpha}|\leq \eta p^{1-\alpha}$. Thus
    \begin{align}\label{eq:regular_bound}
      &\left|\int_0^{p_0} p^{j-1}(1-p)^{n-j-1}\,\Lambda(dp) - A_\Lambda \int_0^{p_0} p^{j-\alpha}(1-p)^{n-j-1}\,dp\right|\nonumber \\
      &\quad\leq \eta A_\Lambda\int_0^{p_0} p^{j-\alpha}(1-p)^{n-j-1}\,dp \leq \eta A_\Lambda \int_0^{1} p^{j-\alpha}(1-p)^{n-j-1}\,dp.
    \end{align}
    Then combining \eqref{eq:gamma_remind} and \eqref{eq:regular_bound} we have
    \begin{align}\label{eq:gamma_estimate}
      &\left| \gamma_{n,j+1}- \binom{n}{j+1}A_\Lambda \int_0^{1} p^{j-\alpha}(1-p)^{n-j-1}\,dp.\right| \nonumber\\
      &\leq \eta \binom{n}{j+1}A_\Lambda \int_0^{1} p^{j-\alpha}(1-p)^{n-j-1}\,dp + \int_{p_0}^1p^{j-1}(1-p)^{n-j-1} |A_\Lambda p^{1-\alpha}-f(p)|\, dp\nonumber\\
      &\leq \eta \binom{n}{j+1}A_\Lambda \int_0^{1} p^{j-\alpha}(1-p)^{n-j-1}\,dp + (1-p_0)^{n-j-1}(A_\Lambda p_0^{1-\alpha} + \Lambda[0,1])
    \end{align}

    From the definition of the Beta function we have that
    \begin{equation}\label{eq:defn_beta}
      \binom{n}{j+1}\int_0^{1} p^{j-\alpha}(1-p)^{n-j-1} \, dp= \frac{\Gamma(j-\alpha+1)\Gamma(n+1)}{\Gamma(j+2)\Gamma(n-\alpha+1)}.
    \end{equation}

    Thus using \eqref{eq:gamma_estimate} and \eqref{eq:defn_beta} we have,
    \begin{align}\label{eq:final_comp_i_hope}
      &\E\left[\sup_{r\in K}\left|\frac{\gamma_{N^\epsilon(r),j+1}}{N^\epsilon(r)^{\alpha}} - A_\Lambda \frac{\Gamma(j-\alpha+1)\Gamma(N^\epsilon(r)+1)}{N^\epsilon(r)^{\alpha}\Gamma(j+2)\Gamma(N^\epsilon(r)-\alpha+1)}\right|^2\right]\nonumber\\
      &\qquad\qquad \leq \eta A_\Lambda \E\left[ \sup_{r \in K}\left| \frac{\Gamma(j-\alpha+1)\Gamma(N^\epsilon(r)+1)}{N^\epsilon(r)^{\alpha}\Gamma(j+2)\Gamma(N^\epsilon(r)-\alpha+1)}\right|^2\right] \nonumber\\
      &\qquad\qquad\qquad + (A_\Lambda p_0^{1-\alpha} + \Lambda[0,1])^2\E[\sup_{r \in K}(1-p_0)^{2(N^\epsilon(r)-j-1)}].
    \end{align}
    The final term on the right hand side converges to $0$ as $\epsilon \rightarrow 0$. For the penultimate term an application of Stirling's formula yields that
    \[
      \E\left[ \sup_{r \in K}\left| \frac{\Gamma(j-\alpha+1)\Gamma(N^\epsilon(r)+1)}{N^\epsilon(r)^{\alpha}\Gamma(j+2)\Gamma(N^\epsilon(r)-\alpha+1)} - \frac{\Gamma(j-\alpha+1)}{\Gamma(j+2)}\right|^2\right] \rightarrow 0
    \]
    as $\epsilon\rightarrow 0$. Thus the term on the left hand side of \eqref{eq:final_comp_i_hope} goes to zero as $\epsilon \rightarrow 0$. Moreover using the triangle inequality we have that
    \begin{align}\label{eq:triangle_beta}
      &\E\left[\sup_{r\in K}\left|\frac{\gamma_{N^\epsilon(r),j+1}}{N^\epsilon(r)^{\alpha}} - A_\Lambda\frac{\Gamma(j-\alpha+1)}{\Gamma(j+2)}\right|^2\right]\nonumber\\
      &\qquad \leq \E\left[\sup_{r\in K}\left|\frac{\gamma_{N^\epsilon(r),j+1}}{N^\epsilon(r)^{\alpha}} - A_\Lambda \frac{\Gamma(j-\alpha+1)\Gamma(N^\epsilon(r)+1)}{N^\epsilon(r)^{\alpha}\Gamma(j+2)\Gamma(N^\epsilon(r)-\alpha+1)}\right|^2\right] \\
      &\qquad\qquad + \E\left[ \sup_{r \in K}\left| A_\Lambda\frac{\Gamma(j-\alpha+1)\Gamma(N^\epsilon(r)+1)}{N^\epsilon(r)^\alpha\Gamma(j+2)\Gamma(N^\epsilon(r)-\alpha+1)} - \frac{\Gamma(j-\alpha+1)}{\Gamma(j+2)}\right|^2\right] \nonumber\\
      &\qquad \rightarrow 0\nonumber
    \end{align}
    as $\epsilon \rightarrow 0$. A final application of the triangle inequality and using \eqref{eq:speed_estimate} gives
    \begin{align*}
      &\E\left[\sup_{r\in K}\left|\epsilon\frac{\gamma_{N^\epsilon(r),j+1}}{N^\epsilon(r)}-A_\Lambda \frac{\Gamma(2-\alpha)\Gamma(j-\alpha+1)}{(1-r)\alpha\Gamma(j+2)}\right|\right]\\
      & \leq \E\left[\sup_{r\in K}\left|\epsilon\frac{\gamma_{N^\epsilon(r),j+1}}{N^\epsilon(r)}-\frac{\alpha}{(1-r)\Gamma(2-\alpha)}\frac{\gamma_{N^\epsilon(r),j+1}}{N^\epsilon(r)^{\alpha}}\right|\right]+ \E\left[\sup_{r\in K}\left|\frac{\gamma_{N^\epsilon(r),j+1}}{N^\epsilon(r)^{\alpha}} - \frac{\Gamma(j-\alpha+1)}{\Gamma(j+2)}\right|\right].
    \end{align*}
    The proof now follows from \eqref{eq:triangle_beta} and \eqref{eq:final_comp_i_hope}.

  \end{proof}

  \begin{rem}
    In the case when $\alpha =2$ all the arguments in this section apply apart from Lemma \ref{lemma:limit_Y}. In its place we have that for $\alpha=2$ as $\epsilon \rightarrow 0$,
    \[
      \E\left[\sup_{r\in K}\left|\epsilon\frac{\gamma_{N^\epsilon(r),j+1}}{N^\epsilon(r)}-\frac{1}{1-r}\1_{j=1}\right|\right]\rightarrow 0.
    \]
    Indeed, this follows from the fact that for $\alpha=2$, $\gamma_{n,2}=\binom{n}{2}$ and $\gamma_{n,j+1}=0$ for $j>1$, together with \cite[Theorem 2]{comingdownlambda}.
  \end{rem}

  The preceding two lemmas directly imply that for each $M>1$, $j \leq M$ and $K \subset [0,1)$, there exists a constant $C=C(M,K)>0$ such that
    \begin{align}\label{eq:qualitative_beta}
      \limsup_{\delta\rightarrow 0}&\left|\frac{1}{\delta}\P(A;\Z^M_\epsilon(r+\delta)-\Z^M_\epsilon(r)=j)\right.\left.- A_\Lambda \frac{\Gamma(2-\alpha)\Gamma(j-\alpha+1)}{(1-r)\alpha\Gamma(j+2)}\E\left[(j+\Z^M_\epsilon(r))\1_A\right] \right| \leq C \epsilon
    \end{align}
  uniformly over all $r \in K$ and $A \in \F^\epsilon_r$.

  For $f :\N\rightarrow \R$ recall that
  \begin{equation}\label{eq:generator_Z}
    L_rf(i):=A_\Lambda\sum_{j\geq 1}(i+j)\frac{\Gamma(2-\alpha)\Gamma(j-\alpha+1)}{(1-r)\alpha\Gamma(j+2)}  [f(i+j)-f(i)].
  \end{equation}
  Using the last two lemmas we are able to show that $\Z^M_\epsilon$ almost solves a martingale problem. This will enable us to show that the limiting process satisfies the martingale problem.

  \begin{lemma}\label{lemma:martingale_close}
    Let $u<r \in [0,1)$ then for any $f:\N\rightarrow \R$ with support in $\{1,\dots,\lfloor M\rfloor\}$,
    \begin{align*}
      &\lim_{\epsilon \rightarrow 0}\sup_{A \in \F^\epsilon_u}\left | \E\left[\left(f(\Z^M_{\epsilon}(r)) -\int_0^r L_sf(\Z^M_\epsilon(s))\, ds\right)\1_{A} \right] - \E\left[\left(f(\Z^M_{\epsilon}(u)) -\int_0^u L_sf(\Z^M_\epsilon(s))\, ds\right)\1_A\right] \right|=0.
    \end{align*}
  \end{lemma}
  \begin{proof}
    Fix $u<r \in [0,1)$, $A \in \F^\epsilon_s$ and $f:\N\rightarrow \R$ with support in $\{1,\dots,\lfloor M\rfloor\}$. Suppose that $\delta>0$ is small. Suppose that $\{s_\ell\}_{i=0}^m$ is such that $s_0=u$, $s_m=r$ and $s_\ell-s_{\ell-1}\leq \delta$ for each $\ell=1,\dots,m$. Now
    \begin{align}\label{eq:mgale_sum}
      &\E[(f(\Z^M_\epsilon(r))-f(\Z^M_\epsilon(u)))\1_A ] = \sum_{\ell=1}^m \E[(f(\Z^M_\epsilon(s_\ell))-f(\Z^M_\epsilon(s_{\ell-1})))\1_A] \nonumber\\
      &\quad =\sum_{\ell=1}^m \sum_{j=1}^M \sum_{i=1}^M [f(i+j)-f(i)]\P(\Z^M_\epsilon(s_\ell)-\Z^M_\epsilon(s_{\ell-1})=j; \Z^M_\epsilon(s_{\ell-1})=i ; A).
    \end{align}

    Now by \eqref{eq:qualitative_beta}, for small enough $\delta>0$, there exists a constant $C>0$ independent of $i, j, \ell$ and $A$ (but depending on $M$) such that
    \begin{align}\label{eq:rate_bound}
      &\left | [f(i+j)-f(i)]\frac{\P(\Z^M_\epsilon(s_\ell)-\Z^M_\epsilon(s_{\ell-1})=j; \Z^M_\epsilon(s_{\ell-1})=i ; A)}{\delta}- L_{s_{\ell-1}}f(i)\P(\Z^M_\epsilon(s_{\ell-1})=i ;A)\right|\leq C \epsilon.
    \end{align}
    Suppose also that $\delta>0$ is small enough so that
    \begin{align}\label{eq:riemann_sum}
      &E\left[\left | \int_u^r L_sf(\Z^M_\epsilon(s))\, ds - \delta \sum_{\ell=1}^{m} L_{s_{\ell-1}}f(\Z^M_\epsilon(s_{\ell-1}))\right| \right] <\epsilon.
    \end{align}
    Thus from \eqref{eq:mgale_sum}, \eqref{eq:rate_bound} and \eqref{eq:riemann_sum} and using the fact that $f$ is bounded,
    \[
      \left|\E\left[\left(f(\Z^M_\epsilon(r))-f(\Z^M_\epsilon(u)) - \int_u^r L_sf(\Z^M_\epsilon(s))\, ds\right)\1_A \right] \right| \leq C'\epsilon
    \]
    for some new constant $C'>0$. The result follows by taking limits.
  \end{proof}

  Next we show a tightness result.

  \begin{lemma}\label{lemma:tightness}
    For each $M>1$, the sequence of processes $\{\Z^M_\epsilon\}_{\epsilon>0}$ is tight in the Skorokhod sense.
  \end{lemma}
  \begin{proof}
    Let $M>1$ be fixed. To prove this lemma we will verify the conditions of \cite[Corollary 2]{aldous_tightness}. Note that $\Z_\epsilon^M$ is uniformly bounded by $M$. Hence it suffices to check that for each $s \in [0,1)$, there exists a deterministic constant $\alpha(\epsilon,\delta)$, possibly depending on $s$, such that
    \begin{equation}\label{eq:aldous_tightness}
      \sup_{r \leq s}\P(\Z_\epsilon^M(r+\delta) -\Z_\epsilon^M(r)>0| \F_r^\epsilon)\leq \alpha(\epsilon,\delta) \qquad \text{ a.s.}
    \end{equation}
    and
    \begin{equation}\label{eq:aldous_alpha}
      \lim_{\delta \rightarrow 0}\limsup_{\epsilon \rightarrow 0}\alpha(\epsilon,\delta)=0.
    \end{equation}

    Fix $s \in [0,1)$, let $r \leq s$ and let $\delta >0$. Let $M^2_\epsilon(a,b)$ denote the event that there are two or more coalescent events during the time interval $((1-a)\epsilon,(1-b)\epsilon)$. Then by \eqref{eq:sum_me_with_n} there exists a constant $C>0$ such that gives that for each $A \in \F_r^\epsilon$,
    \begin{align*}
      \P&(A;\Z_\epsilon^M(r+\delta) -\Z_\epsilon^M(r)>0) \\
      &\leq \sum_{\ell = 1}^M \sum_{j=1}^\ell \sum_{n =j}^\infty \P(A;Z_\epsilon(r+\delta)=\ell+j ; \Z_\epsilon(r)=\ell ; N^\epsilon(r+\delta)=n|M^2_\epsilon(r,r+\delta)^c) + \P(M^2_\epsilon(r,r+\delta))\\
      & \leq C\delta\epsilon \sum_{\ell = 1}^M \sum_{j=1}^\ell \sum_{n =j}^\infty \frac{j+\ell}{n} \gamma_{n,j+1} \P(N^\epsilon(r)=n|M^2_\epsilon(r,r+\delta)^c) + \P(M^2_\epsilon(r,r+\delta))\\
      &\leq C' \delta  \sum_{j=1}^M\E\left[\epsilon\frac{\gamma_{N^\epsilon(r),j+1}}{N^\epsilon(r)}\right]
    \end{align*}
    where $C'=C'(M,s)>0$ is a constant and in the final inequality we have used the fact that $\sup_{r \leq s}\P(M^2_\epsilon(r,r+\delta))=o(\delta)$.

    Hence we have that \eqref{eq:aldous_tightness} holds with
    \[
      \alpha(\epsilon,\delta)= \delta C' \sup_{r \leq s}\sum_{j=1}^M \E\left[\epsilon\frac{\gamma_{N^\epsilon(r),j+1}}{N^\epsilon(r)}\right].
    \]
    Using Lemma \ref{lemma:limit_Y},
    \begin{align*}
      \limsup_{\epsilon \rightarrow 0} \alpha(\epsilon,\delta) & = \delta C' \limsup_{\epsilon \rightarrow 0}\, \sup_{r \leq s}\sum_{j=1}^M \E\left[\epsilon\frac{\gamma_{N^\epsilon(r),j+1}}{N^\epsilon(r)}\right]  \\
      & = \delta C' \sum_{j=1}^M A_\Lambda \frac{\Gamma(2-\alpha)\Gamma(j-\alpha+1)}{(1-s)\alpha\Gamma(j+2)}.
    \end{align*}
    Taking limits as $\delta \rightarrow 0$ in the above equation we obtain \eqref{eq:aldous_alpha}.
  \end{proof}

  Now we can show the convergence of $\Z_\epsilon$ in the Skorokhod sense to a Markov process $\Z=(\Z(r): r \in [0,1))$ with generator given by \eqref{eq:generator_Z}.

  \begin{proof}[Proof of Theorem \ref{thm:path_back}]
    We are done if we can show that for any $M>1$ we have that
    \[
      \Z^M_\epsilon \rightarrow (\Z(r) \wedge M: r \in [0,1))
    \]
    in the Skorokhod sense as $\epsilon \rightarrow 0$.

    Fix $M>1$. The sequence of processes $\{\Z_\epsilon^M\}_{\epsilon>0}$ is tight by Lemma \ref{lemma:tightness}. Suppose now that for some sequence $\epsilon' \rightarrow 0$ we have that
    \[
      \Z^M_{\epsilon'} \rightarrow \Z^M
    \]
    in the Skorokhod sense as $\epsilon' \rightarrow 0$, to some process $\Z^M=(\Z^M(r):r \in [0,1))$. It is enough to show that $\Z^M$ has the same law as $(\Z(r) \wedge M: r \in [0,1))$. We will show this by showing that $\Z^M$ satisfies a martingale problem. Let $f:\N\rightarrow \R$ have support in $\{1,\dots,\lfloor M\rfloor\}$. Then to prove the lemma, it is enough to show that
    \begin{equation}\label{eq:gw_mgale_problem}
      M^f_r:=f(\Z^M(r))-\int_0^r L_s f(\Z^M(s))\, ds
    \end{equation}
    is a martingale. Let $u\in [0,1)$ be fixed and let $\mathbb{D}([0,u],\N)$ denote the Skorokhod space of c\' adl\' ag functions $g:[0,u]\rightarrow \N$. Suppose that $F:\mathbb{D}([0,u],\N)\rightarrow \R$ is a continuous and bounded function. We will show \eqref{eq:gw_mgale_problem} by showing that for $u<r<1$,
    \begin{equation}\label{eq:beta_mgale_property}
      \E[M^f_r F((\Z^M(s):s\leq u))]=\E[M^f_u F((\Z^M(s):s\leq u))].
    \end{equation}

    Fix $u<r<1$ and $\eta>0$. The Skorokhod convergence implies that there exist an $\epsilon_0>0$ such that
    \begin{align}\label{eq:beta_mgale_skorohod1}
      &\bigg | \E[M^f_r F((\Z^M(s):s\leq u))]  \left. - \E\left[\left(f(\Z^M_{\epsilon_0}(r)) -\int_0^r L_s f(\Z^M_{\epsilon_0}(s))\, ds\right) F((\Z^M_{\epsilon_0}(s):s\leq u))\right]\right|<\eta.
    \end{align}

    On the other hand by Lemma \ref{lemma:martingale_close}, we have that there exists an $\epsilon_1>0$ such that
    \begin{align}\label{eq:beta_mgale_skorohod2}
      &\left | \E\left[\left(f(\Z^M_{\epsilon_1}(r)) -\int_0^r L_s f(\Z^M_{\epsilon_1}(s))\, ds\right) F((\Z^M_{\epsilon_1}(s):s\leq u))\right] \right. \nonumber \\
      &\qquad - \left.\E\left[\left(f(\Z^M_{\epsilon_1}(r)) -\int_0^r L^{\epsilon_1}_s f(\Z^M_{\epsilon_1}(s))\, ds\right) F((\Z^M_{\epsilon_1}(s):s\leq u))\right]\right|<\eta.
    \end{align}

    Applying the Skorokhod convergence once more yields that there exists an $\epsilon_2>0$ such that
    \begin{align*}
      &\bigg | \E[M^f_u F((\Z^M(s):s\leq u))] \left. - \E\left[\left(f(\Z^M_{\epsilon_2}(u)) -\int_0^u L_s f(\Z^M_{\epsilon_2}(s))\, ds\right) F((\Z^M_{\epsilon_2}(s):s\leq u))\right]\right|<\eta.
    \end{align*}
    Combining this with \eqref{eq:beta_mgale_skorohod1} and \eqref{eq:beta_mgale_skorohod2} gives that
    \[
      |\E[M^f_r F((\Z^M(s):s\leq u))]-\E[M^f_u F((\Z^M(s):s\leq u))]|<3\eta
    \]
    As $\eta>0$ is arbitrary this shows \eqref{eq:beta_mgale_property} which concludes the proof.
  \end{proof}

  \subsection{Proof of Theorem \ref{thm:main}}\label{sec:proof_a_ranged}

  The proof of Theorem \ref{thm:main} follows from Theorem \ref{thm:path_back} in a straight forward manner. We first begin by describing a sufficient and necessary condition for the convergence of compact ultra--metric spaces. Let recall that a metric space $(X,d)$ is called and ultra--metric space if
  \[
  d(x,y) \leq d(x,z) \vee d(z,y) \qquad \forall x,y,z \in X.
  \]
  Compact ultra--metric spaces have the property that for any $\eta>0$, there exists disjoint by closed balls $B_1,\dots,B_k$ of radius $\eta$ which cover $X$.

  \begin{defn}
    For a compact ultra--metric space $(X,d)$, the $\eta$-cover space $(X^\eta,d^\eta)$ is defined as follows. Let $B_1,\dots,B_k$ be the disjoint closed balls of radius $\eta$ which cover $X$. Then $X^\eta=\{1,\dots,k\}$ and
    \[
    d^\eta(i,j) : =\inf\{d(x,y): x\in B_i, y\in B_j\} \qquad i,j \in X^\eta.
    \]
  \end{defn}

  The following lemma is not hard to show and we leave the proof out.

  \begin{lemma}\label{lemma:cover_space}
    Let $(X_1,d_1),(X_2,d_2),\dots$ be a sequence of compact ultra--metric spaces and for $\eta>0$ let $(X^\eta_1,d^\eta_1),(X^\eta_2,d^\eta_2),\dots$ denote the respective $\eta$-cover spaces.

    Suppose that for every $\eta>0$, $(X^\eta_k,d^\eta_k) \rightarrow (X^\eta,d^\eta)$ under the compact Gromov--Hausdorff metric as $k\rightarrow \infty$, then there exists a compact ultra--metric space $(X,d)$ such that $\lim_{k \rightarrow \infty}(X_k,d_k) =(X,d)$ in the compact Gromov--Hausdorff sense.
  \end{lemma}

  Henceforth fix $\eta\in (0,1)$. Here and throughout we let $B_\epsilon(i,r)$ denote the closed ball of radius $r$ around $i$ in the space $(E,\epsilon^{-1}\delta)$. Then to show the theorem it suffices to show that $(B_\epsilon(1,1),\epsilon^{-1}\delta)$ converges weakly. Indeed we may assume that $i=1$ by exchangeability of the coalescent and the proof for general $r>0$ follows with more cumbersome notation.

  Notice that $(B_\epsilon(1,1),\epsilon^{-1}\delta)$ is a compact ultra--metric space. Let $(S_\epsilon,r_\epsilon)$ denote its $\eta$-cover space. Thus to show the theorem, it suffices to verify Lemma \ref{lemma:cover_space} by showing the convergence of $(S_\epsilon,r_\epsilon)$ as $\epsilon \downarrow 0$.

  The cardinality of $S_\epsilon$ is precisely the number of blocks of $\Pi$ at time $(1-\eta)\epsilon$ that make up the block containing $1$ at time $\epsilon$ and hence $|S_\epsilon|=\mathcal Z_\epsilon(1-\eta)$. Furthermore we will now see that we can discover the exact structure of $(S_\epsilon,r_\epsilon)$ using the process $\Z_\epsilon$.

  Let $B_1,\dots,B_{\Z_{\epsilon}(1-\eta)}$ denote the disjoint closed balls of radius $\eta$ that cover the space $(B_\epsilon(1,1),\epsilon^{-1}\delta)$. Without a loss of generality we will assume that $1 \in B_1$.
  Given the process $\Z_\epsilon$ we can construct the space $(S_\epsilon,r_\epsilon)$ as follows. First construct a tree $T$ from a branching process. Start the tree with one immortal particle which will not die. We call all other particles mortal. Suppose now that for some $i \in \N$ and $j \geq 1$, $\Z_\epsilon(r-)=i$ and $\Z_\epsilon(r+)=i+j$. Then there is a birth at height $r$ of the tree. With probability $(j+1)/(i+j)$, the immortal particle gives birth to $j$ offspring and with probability $(i-1)/(i+j)$ a uniformly chosen mortal individual dies gives birth to $j+1$ offspring. Thus the tree has exactly $\Z_\epsilon(1-\eta)$ leaves and height $1-\eta$. These leaves form the space $S_\epsilon$ and the distance $r_\epsilon$ between two leaves is the genealogical distance i.e. half of the length of the unique path between the two leaves. Use the same procedure but with the process $\Z$ to obtain a space $(S,r)$. It is clear that $(S,r)$ is a compact metric space.

  Using Theorem \ref{thm:path_back} and Skorokhod's representation theorem suppose henceforth that $\Z_\epsilon \rightarrow \Z$ as $\epsilon \rightarrow 0$ almost surely under the Skorokhod topology. Let $J^\epsilon_1,\dots, J^\epsilon_n$ be the jumps of the process $\Z_\epsilon$ before time $1-\eta$ and similarly let $J_1,\dots, J_m$ be the jumps of the process $\Z$ before time $1-\eta$. Then from the almost sure convergence of the process $\Z_\epsilon$ to $\Z$ we can conclude the following for $\epsilon>0$ small enough. Firstly $m=n$ and $Z_\epsilon(J^\epsilon_i)=Z(J_i)$ for each $i \leq n=m$. Second, $\max_{i\leq n}|J^\epsilon_i-J_i|$ is small.

  Thus for $\epsilon>0$ small enough this gives a coupling between the spaces $(S_\epsilon,r_\epsilon)$ and $(S,r)$ such that $S=S_\epsilon$ and further
  \[
    \max_{i,j \in S}|r_\epsilon(i,j)-r(i,j)| \leq 2\Z_\epsilon(1-\eta) \max_{i\leq n}|J^\epsilon_i-J_i|
  \]
  which is small. Hence $(S_\epsilon,r_\epsilon)\rightarrow (S,r)$ almost surely under the compact Gromov--Hausdorff topology as $\epsilon \rightarrow 0$ and the theorem now follows from Lemma \ref{lemma:cover_space}.

  \section{Kingman Case}\label{sec:proof_a=2}
  In this section we will prove Theorem \ref{thm:mm-converge}. As before we write $B_\epsilon(i,r)$ to mean the closed ball of radius $r$ around $i$ in the space $(E,\epsilon^{-1}\delta)$. Recall the construction of the space $(\mathbb S,d_\mathbb S, \mu, 0)$ given in the statement of Theorem \ref{thm:mm-converge}. Again we will only show that
  \[
    (B_\epsilon(1,1),\epsilon^{-1}\delta, 4\epsilon^{-1}\nu) \rightarrow (B(0,1), d_\mathbb S, \mu)
  \]
  weakly under the compact Gromov--Hausdorff metric as $\epsilon \rightarrow 0$, where $B(0,1)\subset (\mathbb S, d_\mathbb S)$ is the closed ball of radius $1$ around $0$.

  \begin{figure}
    \begin{center}
      \includegraphics[scale=1]{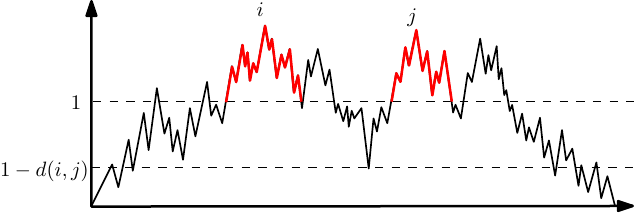}
      \caption{Visual interpretation of the construction of the metric space $(S,d)$.}\label{fig:tilde_E_construction}
    \end{center}
  \end{figure}

  We first show how to construct metric measure spaces using excursions. We term this the Evans metric space associated to an excursion due to the similarities of the Evans metric space associated to coalescent processes. We describe this process in generality and then use it to construct the spaces $(B_\epsilon(1,1),\epsilon^{-1}\delta, 4\epsilon^{-1}\nu)$, $(B(0,1), d_\mathbb S, \mu)$ as well as an auxiliary space.

  \subsubsection*{Constructing Evans metric measure space from an excursion}

  Let $e=(e(t):0 \leq t \leq \zeta(e))$ be an excursion that has height greater than $1$ meaning $e:[0,\zeta(e)] \rightarrow [0,\infty)$ is a continuous path such that $e$ hits $1$ and further $e(t)=0$ if and only if $t \in \{0,\zeta(e)\}$. In all the applications we consider, $e$ will be Brownian. In this case there exists a jointly continuous local time $(L(t,x):t \in [0,\zeta(e)],x \in \R)$ with the property that
  \[
    L(t,x)= \lim_{\epsilon \rightarrow 0}\frac{1}{2\epsilon}\int_0^t \1_{\{e(s) \in (x-\epsilon,x+\epsilon)\}}ds \qquad t \in [0,\zeta(e)], x \in \R,
  \]
  see for example \cite[Chapter IV]{revuzyor} and in particular Corollary 1.8 and Corollary 1.9. We now show how to obtain a metric measure space $(S,d,\pi)$ from the excursion $e$. Let $\{\epsilon_i\}_{i=1}^\infty$ be the positive excursions of the excursion $e$ above level $1$ and we order them as follows. Since $t \mapsto L(t,1)$ is increasing the Stieltjes measure $dL(\cdot,1)$ exists and we let $U_1,U_2\dots$ be i.i.d. random variables with the law given by $dL(\cdot,1)/Z_1$ where $Z_1=L(\zeta(e), 1)$. For each $k \in \N$, $U_k$ is the local time corresponding to a unique excursion at level $1$ (see \cite[Chapter IV, Proposition 2.5]{revuzyor}) which may be positive or negative. Let $U_{k_1}, U_{k_2},\dots$ denote the local times corresponding to positive excursions, then $\{\epsilon_i\}_{i=1}^\infty$ is ordered such that for each $i \in \N$, $\epsilon_i$ is the excursion which starts at local time $U_{k_i}$.

  For $i,j \in \N$ with $i \neq j$ we define $1-d(i,j)$ to be the first height at which $\epsilon_i$ and $\epsilon_j$ are a part of the same excursion (see Figure \ref{fig:tilde_E_construction}). In other words let $t(\epsilon_i)$ and $t(\epsilon_j)$ denote the start time of the excursions $\epsilon_i$ and $\epsilon_j$ and suppose without a loss of generality that $t(\epsilon_i)<t(\epsilon_j)$. Then
  \begin{equation}\label{eq:metric_excursion}
    d(i,j)=1- \inf\{e(t): t(\epsilon_i) \leq t \leq t(\epsilon_j)\}.
  \end{equation}
  By definition, the space $(S,d)$ is the completion of $(\N,d)$. We also define a measure $\pi$ on $(S,d)$ as follows. For each $i\in \N$ and $r \in (0,1]$, every closed ball $B(i,r) \subset (S,d)$ corresponds to an excursion $e$ of $e$ above level $1-r$ that hits level $1$. We define $\pi(B(i,r))$ to be the total local time $\ell_1(e)$ the excursion $e$ attains at level $1$. Note that this defines the measure uniquely by Carath\'edeory's extension theorem and thus we obtain a metric measure space $(S,d,\pi)$. We remark that the total mass $\pi(S)=Z_1$ is the total local time spent at level $1$ by the excursion $e$ and thus is finite. Lastly in all of our applications the excursion $e$ will have unique local minima which is enough to conclude that $(S,d)$ is compact.

  \begin{defn}
    The metric measure space $(S,d,\pi)$ is called the Evans metric measure space associated to the excursion $e$.
  \end{defn}
  \begin{rem}
    For the reader who is familiar with continuum real trees, the Evans metric measure space associated to the excursion $e$ can be constructed as follows. Let $(T,d_T, o)$ be the rooted real tree that is encoded by the excursion $e$ where $o \in T$ is the root of the tree. Then $S$ is the set of points $x \in T$ such that $d_T(o,x)=1$, and the metric $d$ is given by $d=(1/2)d_T$. Let $\tilde T$ be the subtree of $T$ spanned by $S$ and the root $o$ so that $S$ is the let of leaves of $\tilde T$. Then $\pi$ is the uniform measure on $(\tilde T, d_T)$ which is supported on $S$.
  \end{rem}

  \subsubsection*{Constructing $(E,\delta,\nu)$ from an excursion}

  Let $X=(X_t:0 \leq t \leq \zeta(X))$ be a Brownian excursion conditioned to hit level $1$. Let $(\tilde E, \tilde \delta, \tilde \nu)$ be the Evans metric measure space associated to the excursion $X$. Then \cite[Theorem 1.1]{natkingbrown} gives a construction of $(E,\delta,\nu)$ in terms of $(\tilde E, \tilde \delta, \tilde \nu)$ which we now describe. For $x \geq 0$ let $Z_x$ be the total local time attained at level $x\geq 0$ by the process $X$. For $t \in [0,1]$ define
  \begin{align}\label{eq:time_change}
    V(t):= \int_{1-t}^1 \frac{4}{Z_v} dv.
  \end{align}
  Then $E= \tilde E$ and for $x,y \in E$, $\delta(x,y)=V(\tilde \delta(x,y))$. Lastly $\nu$ is the renormalisation of $\tilde \nu$, that is $\nu(\cdot)=\tilde \nu(\cdot)/Z_1$ where $Z_1$ is the total local time that the excursion $X$ attains at level $1$.

  \subsubsection*{Constructing an intermediary space $(\tilde B_{\epsilon}(1,1),T_\epsilon \tilde \delta,T_\epsilon\tilde\nu)$}

  For $\epsilon>0$ let
  \begin{equation}\label{eq:T_eps_defn}
    T_\epsilon:=\frac{4}{\epsilon Z_{1-\sqrt \epsilon}} \vee \frac{1}{\sqrt \epsilon}.
  \end{equation}
  The maximum in the definition is for technical reasons and for small $\epsilon>0$ it will be the case that $T_\epsilon=4(\epsilon Z_{1-\sqrt \epsilon})^{-1}$ with high probability. Thus for small $\epsilon>0$, it is not hard to see that $V(t\epsilon)\approx t T_\epsilon$. We will now construct a space $(\tilde B_{\epsilon}(1,1),T_\epsilon \tilde \delta,T_\epsilon\tilde\nu)$ which will turn out to be close to $(B_\epsilon(1,1),\epsilon^{-1}\delta, 4\epsilon^{-1}\nu)$ and has the advantage that it is easier to show convergence to the limiting space $(B(0,1), d_\mathbb S,\mu)$.

  Let $\tilde B_\epsilon(1,1) \subset (\tilde E, \tilde \delta)$ be the closed ball of radius $1/T_\epsilon$ around $1$ and consider $(\tilde B_{\epsilon}(1,1),T_\epsilon \tilde \delta,T_\epsilon\tilde\nu)$ as a compact metric measure space in its own right. Thus the space $(\tilde B_{\epsilon}(1,1),T_\epsilon \tilde \delta,T_\epsilon\tilde\nu)$ has diameter $1$.
  It will be convenient for us to obtain this space as the Evans metric measure space of some excursion. Since $(\tilde E, \tilde\delta,\tilde\nu)$ is obtained from a Brownian excursion $X$ conditioned to reach level $1$, we can obtain $(\tilde B_\epsilon(1,1),T_\epsilon \tilde \delta,T_\epsilon\tilde\nu)$ by scaling $X$. To be more precise, define
  \begin{equation}\label{eq:X_eps_defn}
    X_\epsilon(t):=1+[X(tT_\epsilon^{-2})-1]T_\epsilon \qquad\qquad 0 \leq t \leq \zeta(X)T_\epsilon^2.
  \end{equation}
  Then $(\tilde B_\epsilon(1,1),T_\epsilon \tilde \delta,T_\epsilon\tilde\nu)$ is the Evans metric measure space obtained using a sub-excursion $Y_\epsilon$ of $X$ above level $0$ that hits level $1$.

  To describe the law of $Y_\epsilon$ fully, let $e^{(\epsilon)}_1,\dots, e^{(\epsilon)}_{M(\epsilon)}$ be the excursions of $X_\epsilon$ above level $0$ that reach level $1$ where $M(\epsilon)$ is the total number of such excursions. Though we have obtained the process $X_\epsilon$ by performing diffusive scaling on $X$, the scaling factor $T_\epsilon$ is random so it is not obvious that $e^{(\epsilon)}_1,\dots, e^{(\epsilon)}_{M(\epsilon)}$ are themselves Brownian excursions conditioned to reach level $1$. We will see that this is nevertheless the case in Lemma \ref{lemma:excursion_brownian}. From the construction of the Evans metric measure space above we see that given the excursions $e^{(\epsilon)}_1,\dots, e^{(\epsilon)}_{M(\epsilon)}$, the excursion $Y_\epsilon$ is selected by biasing on the total local time obtained at level $1$. Precisely, $Y_\epsilon$ can be described by the following
  \begin{equation}\label{eq:Y_eps_defn}
    \P(Y_\epsilon=e^{(\epsilon)}_i | e^{(\epsilon)}_1,\dots, e^{(\epsilon)}_{M(\epsilon)} ) = \frac{\ell_1(e^{(\epsilon)}_i)}{\sum_{j=1}^{M(\epsilon)} \ell_1(e^{(\epsilon)}_j)} \qquad 1 \leq i \leq M(\epsilon)
  \end{equation}
  where $\ell_1(e^{(\epsilon)}_i)$ is the total local time at $1$ that the excursion $e^{(\epsilon)}_i$ accumulates.

  \subsubsection*{Constructing $(B(0,1), d_\mathbb S,\mu)$ from an excursion}

  Finally we show how to construct $(B(0,1), d_\mathbb S,\mu)$ directly from an excursion. Let $W=(W_t:t \in \R)$ be a two-sided Brownian motion and let $Y=(Y_t: t \in [0,\zeta(Y)])$ be the excursion of $W$ above level $-1$ straddling the origin. That is, let $\tau_{+}=\inf\{t>0: W_t=-1\}$ and $\tau_{-}=\sup\{t<0: W_t=-1\}$, then $Y_t=W_{t+\tau_{-}}+1$ for $t \leq \tau_{+}-\tau_{-}$. It is not hard to check that the space $(B(0,1), d_\mathbb S,\mu)$ can be constructed as the Evans metric measure space associated to the excursion $Y$.

  \medskip

  We sum up the various spaces defined here with the following table:
  \renewcommand{\arraystretch}{1.5}
  \begin{center}
    \begin{tabular}{ | l | l | p{10cm} |}
    \hline
    \textbf{Evans mm-space} & \textbf{Excursion} & \textbf{Description of excursion} \\ \hline
    $(\tilde E, \tilde \delta, \tilde \nu)$ & $X$ & Brownian excursion conditioned to reach level $1$\\\hline
    $(\tilde B_{\epsilon}(1,1),T_\epsilon \tilde \delta,T_\epsilon\tilde\nu)$ & $Y_\epsilon$ & Local-time biased pick of excursions of $X_\epsilon$ above level $0$ that hit level $1$, where $ X_\epsilon(t):=1+[X(tT_\epsilon^{-2})-1]T_\epsilon$\\\hline
    $(B(0,1), d_\mathbb S,\mu)$ & $Y$ & Excursion of a two-sided Brownian motion above level $-1$ straddling the origin\\\hline
    \end{tabular}
\end{center}
\renewcommand{\arraystretch}{1}

  \subsubsection*{Outline of the proof}

  The proof consists of (1) showing the spaces $(\tilde B_{\epsilon}(1,1),T_\epsilon \tilde \delta,T_\epsilon\tilde\nu)$ and $(B_\epsilon(1,1),\epsilon^{-1}\delta, 4\epsilon^{-1}\nu)$ are close and (2) showing $\lim_{\epsilon\downarrow 0}(\tilde B_{\epsilon}(1,1),T_\epsilon \tilde \delta,T_\epsilon\tilde\nu) = (B(0,1), d_\mathbb S,\mu)$ in the compact Gromov--Hausdorff--Prokhorov sense. Showing (1) is relatively straight forward and relies on analytical estimates on $V(t)$ as well as well as controls on $Z_x$. We show (1) at the end with equation \eqref{eq:tilde_E_close}.

  Showing (2) on the other hand is a little bit more complicated and will take up large portion of the proof.
  First in Lemma \ref{lemma:excursion_brownian} we show that the excursions that $Y_\epsilon$ is picked from are actually Brownian excursions conditioned to reach level $1$. An important observation is that the excursion $Y$ can be picked from an infinite sequence of i.i.d. Brownian excursions conditioned to reach level $1$, where each excursion is weighted by the total local it accumulates at level $1$. The difference is that $Y_\epsilon$ is picked from a finite set of excursions, however the number of excursions that $Y_\epsilon$ is picked from goes to $\infty$ as $\epsilon \downarrow 0$.
  We will use this in Lemma \ref{lemma:tilde_S_close} to give a suitable coupling of $Y_\epsilon$ to $Y$ which will imply $\lim_{\epsilon\downarrow 0}(\tilde B_{\epsilon}(1,1),T_\epsilon \tilde \delta,T_\epsilon\tilde\nu) = (B(0,1), d_\mathbb S,\mu)$.

  \bigskip

  We begin the proof with the following lemma.

  \begin{lemma}\label{lemma:excursion_brownian}
    Let $e^{(\epsilon)}_1,\dots, e^{(\epsilon)}_{M(\epsilon)}$ be the excursions of $X_\epsilon$ above level $0$ that reach level $1$. Then conditionally on $M(\epsilon)$, $e^{(\epsilon)}_1,\dots, e^{(\epsilon)}_{M(\epsilon)}$ are i.i.d. Brownian excursions conditioned to reach level $1$.
  \end{lemma}
  \begin{proof}
    Fix $\epsilon>0$. Observe that the excursions $e^{(\epsilon)}_1,\dots, e^{(\epsilon)}_{M(\epsilon)}$ correspond to excursions of $X$ above level $u:=1-1/T_\epsilon$ that hit level $1$. Note that since $T_\epsilon \geq 1/\sqrt\epsilon$ we have that $u \geq u_0:=1-\sqrt\epsilon$. Define the $\sigma$-algebra $\mathcal H=\sigma(X_{\alpha(s)}: s \geq 0)$ where
    \[
      \alpha(s):=\inf\left \{t \geq 0: \int_0^t \1_{\{X(v)\leq u_0\}} dv >s\right \}.
    \]
    In words $\mathcal H$ contains all the information about the excursions of $X$ below level $u_0=1-\sqrt\epsilon$. The total local time $Z_{u_0}$ of the process $X$ at level $u_0$ satisfies (see \cite[Chapter VI Corollary (1.9)]{revuzyor})
    \[
      Z_{u_0}= \lim_{\eta \rightarrow 0}\frac{1}{\eta}\int_0^{\zeta(X)} \1_{\{X_s \in (u_0-\eta,u_0]\}} \, ds.
    \]
    Thus $Z_{u_0}$ is measurable with respect to $\mathcal H$ and consequently so is $T_\epsilon$ and $u=1-1/T_\epsilon$.

    It is well known that after hitting level $u_0$, the law of $X$ is that of a Brownian motion started at level $u_0$, killed the first time it hits $0$ and conditioned to reach level $1$ before hitting level $0$. It\^o's description of Brownian motion (\cite[Chapter XII Theorem (2.4)]{revuzyor}) tells us that conditionally on $Z_{u_0}=z$ the excursions of the process $X$ above level $u_0$ form a Poisson point process on the local time interval $[0,z]$, conditioned to have at least one excursion of height grater than $\sqrt\epsilon$. Further, the excursions above level $1- \sqrt \epsilon$ are independent of the excursions below level $1-\sqrt\epsilon$, and hence independent of the $\sigma$-algebra $\mathcal H$.

    On the other hand $Z_{u_0}$ and $u$ are measurable with respect to $\mathcal H$. Thus conditionally on $\mathcal H$ the excursions of the process  $X$ above level $u=1-1/T_\epsilon$ that hit level $1$ are i.i.d. Brownian excursions conditioned to have height greater than $1/T_\epsilon$. By Brownian scaling and \eqref{eq:X_eps_defn} it follows that conditionally on $\mathcal H$ and $M(\epsilon)=m$, $e^{(\epsilon)}_1,\dots, e^{(\epsilon)}_{m}$ are i.i.d. Brownian excursions conditioned to reach level $1$. In other words let $e$ be a Brownian excursion conditioned to reach level $1$ and $\phi_1,\dots,\phi_m$ be continuous bounded functions mapping the set of excursions to $\R$, then we have just shown that
    \[
      \E[\phi_1(e^{(\epsilon)}_1)\dots \phi_m(e^{(\epsilon)}_m)| \mathcal H, M(\epsilon)=m]= \prod_{i=1}^m\E[\phi_i(e)].
    \]
    Taking expectations conditionally on $M(\epsilon)=m$ on both sides above finishes the proof.
  \end{proof}

  The next lemma shows the convergence of the local time at $1$ of $Y_\epsilon$. Recall that $Y$ is the excursion of two-sided Brownian motion straddling the origin which is used in the construction of the space $(B(0,1), d_\mathbb S, \mu)$.

  \begin{lemma}\label{lemma:total_loc_time}
    We have that
    \[
      \ell_1(Y_\epsilon) \rightarrow \ell_1(Y)
    \]
    in distribution as $\epsilon \rightarrow 0$, where $\ell_1(\cdot)$ denotes the total local time attained by the excursion at level $1$.
  \end{lemma}
  \begin{proof}
    Let $e^{(\epsilon)}_1,\dots, e^{(\epsilon)}_{M(\epsilon)}$ be the excursions of $X_\epsilon$ above level $0$ that reach level $1$. Condition on $M(\epsilon)=m$. Then from Lemma \ref{lemma:excursion_brownian} it follows that $e^{(\epsilon)}_1,\dots, e^{(\epsilon)}_{m}$ are i.i.d. Brownian excursions conditioned to reach level $1$. Thus it follows that $\ell_1(e^{(\epsilon)}_1), \dots, \ell_1(e^{(\epsilon)}_m)$ are i.i.d. exponential random variables with parameter $1/2$ (see \cite[Chapter VI Proposition (4.6)]{revuzyor}). Let $E_1,E_2,\dots$ be i.i.d. exponential random variables with parameter $1/2$, then by \eqref{eq:Y_eps_defn} it follows that
    \[
      \P(\ell_1(Y_\epsilon) \in \cdot | M(\epsilon)=m) = \E\left[\frac{m E_1}{\sum_{i=1}^m E_i} \1_{\{E_1 \in \cdot\}}\right].
    \]
    On the other hand $\ell_1(Y)$ has the same law as $E_1+E_2$, the sum of two independent exponential random variables. This is a sized biased exponential random variable and by bounded convergence and the law of large numbers we have
    \[
      \P(\ell_1(Y) \in \cdot)=  \lim_{n \rightarrow \infty}\E\left[\frac{n E_1}{\sum_{i=1}^n E_i} \1_{\{E_1 \in \cdot\}}\right].
    \]
    The lemma now follows from the fact that $M(\epsilon) \rightarrow \infty$ in probability as $\epsilon \rightarrow 0$.
  \end{proof}

  From Lemma \ref{lemma:excursion_brownian} for $Y_\epsilon$ and the definition of $Y$ we can deduce the following. Conditionally on $\ell_1(Y_\epsilon)=\ell_\epsilon$ and $\ell_1(Y)=\ell$, the excursions $Y_\epsilon$ and $Y$ are both Brownian excursions conditioned on attaining total local time $\ell_\epsilon$ and $\ell$, respectively, at level $1$. Thus at this point it is immediate that $Y_\epsilon \rightarrow Y$ in the Skorokhod sense as $\epsilon \rightarrow 0$. Unfortunately, this is not enough to show the weak convergence of the space $(\tilde B_{\epsilon}(1,1),T_\epsilon \tilde \delta,T_\epsilon\tilde \nu)$, under the Gromov--Hausdorff--Prokhorov metric, to the space $( B(0,1),d_\mathbb S, \mu)$. Instead, in the next lemma we construct a coupling between $Y_\epsilon$ and $Y$ under which the paths of the two processes agree up to a time. This will in turn enable us to show that the spaces $(\tilde B_{\epsilon}(1,1),T_\epsilon \tilde \delta,T_\epsilon\tilde \nu)$ and $( B(0,1),d_\mathbb S, \mu)$ are close to each other.

  \begin{lemma}\label{lemma:tilde_S_close}
    We have that
    \[
      (\tilde B_{\epsilon}(1,1),T_\epsilon \tilde \delta, T_\epsilon\tilde \nu)\rightarrow (B(0,1), d_\mathbb S,\mu)
    \]
    weakly as $\epsilon \rightarrow 0$ under the Gromov--Hausdorff--Prokhorov topology.
  \end{lemma}
  \begin{proof}
    We first present a coupling between $Y_\epsilon$ and $Y$. By Lemma \ref{lemma:total_loc_time} and Skorokhod representation theorem we can suppose that $\ell_1(Y_\epsilon)$ and $\ell_1(Y)$ are coupled such that $\ell_1(Y_\epsilon)\rightarrow \ell_1(Y)$ almost surely as $\epsilon \rightarrow 0$. Fix $\epsilon>0$ and condition on $\ell_1(Y_\epsilon)=\ell_\epsilon$ and $\ell_1(Y)=\ell$. Suppose further that $\ell_\epsilon \leq \ell$ (the other case is similar). Note that both $Y_\epsilon$ and $Y$ are Brownian excursions conditioned to have $\ell_\epsilon$ and $\ell$ total local time at level $1$. Hence we can couple $Y_\epsilon$ and $Y$ such that they have the same path until their local time at level $1$ reaches $\ell_\epsilon$.

    Excursions of the process $Y_\epsilon$ above level $1$ correspond to the dense subset $\N$ of of the space $(\tilde B_{\epsilon}(1,1),T_\epsilon \tilde \delta)$. Similarly for $Y$ and $(B(0,1), d_\mathbb S)$. Thus the coupling of the processes $Y_\epsilon$ and $Y$ gives us a coupling of the spaces such that $(\tilde B_{\epsilon}(1,1),T_\epsilon \tilde \delta)\subset (B(0,1), d_\mathbb S)$. Under this coupling of the spaces it is immediate that
    \begin{equation}\label{eq:pr_dist_bound}
      d_{Pr}(T_\epsilon\tilde\nu, \mu) = |\ell-\ell_\epsilon|.
    \end{equation}
    where $d_{Pr}$ denotes the Prokhorov distance.

    Fix $\eta\in (0,1)$ and recall that $d_{H}((\tilde B_{\epsilon}(1,1),T_\epsilon \tilde \delta),(B(0,1), d_\mathbb S))<\eta$ if and only if the $\eta$-enlargement of the space $(\tilde B_{\epsilon}(1,1),T_\epsilon \tilde \delta)$ contains $(B(0,1), d_\mathbb S)$. Clearly it suffices to show the latter condition for the dense subspaces of $(\tilde B_{\epsilon}(1,1),T_\epsilon \tilde \delta)$ and $(B(0,1), d_\mathbb S)$ corresponding to the excursions above level $1$ of the processes $Y_\epsilon$ and $Y$ respectively. Consider the excursions of $Y$ above level $1$ which appear before local time $\ell_\epsilon$. Call these excursions matched; these are also excursions of $Y_\epsilon$ above level $1$. Then $d_{H}((\tilde B_{\epsilon}(1,1),T_\epsilon \tilde \delta),(B(0,1), d_\mathbb S))<\eta$ if and only if every excursion of $Y$ above level $1-\eta$ that hits level $1$ contains a matched excursion. This is the same as the event that every excursion of $Y$ below $1$ with local time in the interval $[\ell_\epsilon,\ell)$ has infimum greater than $-\eta$. Thus from standard excursion theory (see \cite[Chapter XII Exercise (2.10)]{revuzyor}) we have that
    \begin{equation}\label{eq:gromov_dist_bound}
      \P(d_{H}((\bar B_{T_\epsilon}(1,1),T_\epsilon \tilde \delta),(\bar B(0,1), d_\mathbb S))\leq \eta)= e^{-\frac{|\ell-\ell_\epsilon|}{\eta}}.
    \end{equation}

    The equations \eqref{eq:pr_dist_bound} and \eqref{eq:gromov_dist_bound} hold by the same argument when $\ell< \ell_\epsilon$. Hence in conclusion we have constructed a coupling where
    \begin{align*}
      &\P(d_{H}((\tilde B_{\epsilon}(1,1),T_\epsilon \tilde \delta),(B(0,1), d_\mathbb S))+d_{Pr}(T_\epsilon\tilde \nu, \mu)>2\eta) \\
      &\leq \P(d_{H}((\tilde B_{\epsilon}(1,1),T_\epsilon \tilde \delta),(B(0,1), d_\mathbb S))>\eta)+ \P(d_{Pr}(T_\epsilon\tilde \nu, \mu)>\eta) \\
      & = 1-\E\left[\exp\left(-\frac{|\ell_1(Y_\epsilon)-\ell_1(Y)|}{\eta}\right)\right]+\P(|\ell_1(Y_\epsilon)-\ell_1(Y)|>\eta).
    \end{align*}
    Taking the limit as $\epsilon \rightarrow 0$ above and using bounded convergence finishes the proof.
  \end{proof}

  Using the previous lemma we can now prove Theorem \ref{thm:mm-converge}.

  \begin{proof}[Proof of Theorem \ref{thm:mm-converge}]
    For $t \in [0,1]$ let $U(t)$ denote the inverse of $V(t)$ in \eqref{eq:time_change}, that is
    \[
      U(t)= V^{-1}(t)=\inf\left\{s>0:\int_{1-s}^1 \frac{4}{Z_v}\,dv >t\right\}.
    \]
    Fix $\eta\in (0,1)$ and define
    \begin{align*}
      A_\epsilon^\eta&:=\{(1-\eta)Z_{1-\sqrt \epsilon}\leq Z_s \leq (1+\eta) Z_{1-\sqrt \epsilon} \, \forall s \in [1-U(\epsilon),1]\}\\
      B_\epsilon&:=\left\{T_\epsilon = \frac{4}{\epsilon Z_{1-\sqrt\epsilon}}\right\}=\left\{Z_{1-\sqrt\epsilon} \leq \epsilon^{-1/2}\right\}\\
      \mathcal{E}^\eta_\epsilon&:= A^\eta_\epsilon \cap B_\epsilon.
    \end{align*}

    We claim that $\P(\mathcal{E}^\eta_\epsilon) \rightarrow 1$ as $\epsilon \rightarrow 0$. Indeed, $x \mapsto Z_x$ is uniformly continuous (this follows from the Ray-Knight theorems, see \cite[Chapter XI Theorem (2.2)]{revuzyor}). Further it is elementary to check that $\lim_{\epsilon \rightarrow 0}U(\epsilon)=0$. This shows that $\P(A_\epsilon^\eta)\rightarrow 1$ as $\epsilon \rightarrow 0$. The convergence of $\P(B_\epsilon)$ follows from the fact that $Z_{1-\sqrt \epsilon}$ is distributed exponentially with parameter $1/(2-2\sqrt\epsilon)$ (see \cite[Chapter VI, Proposition (4.6)]{revuzyor}). Thus $\P(\mathcal{E}^\eta_\epsilon) \rightarrow 1$ as $\epsilon \rightarrow 0$.

    On the event $\mathcal{E}^\eta_\epsilon$ we have that for each $t \leq \epsilon$,
    \begin{equation}\label{eq:cond_U_bound}
      \frac{U(t)\epsilon T_\epsilon}{1+\eta}=\frac{4U(t)}{(1+\eta)Z_{1-\sqrt\epsilon}}\leq t\leq \frac{4U(t)}{(1-\eta)Z_{1-\sqrt\epsilon}}=\frac{U(t)\epsilon T_\epsilon}{1-\eta}.
    \end{equation}
    Recall that $E=\tilde E$ and that for each $x,y \in E$, $\tilde\delta(x,y)=U(\delta(x,y))$. Hence from \eqref{eq:cond_U_bound} it follows that on the event $\mathcal{E}^\eta_\epsilon$ for any $x,y \in E$ such that $\delta(x,y)\leq \epsilon$ we have
    \begin{equation}\label{eq:pinching}
      \frac{1}{1+\eta}T_\epsilon\tilde\delta(x,y)\leq \epsilon^{-1}\delta(x,y)\leq \frac{1}{1-\eta}T_\epsilon\tilde\delta(x,y).
    \end{equation}
    A brief computation shows that on the event $\mathcal{E}^\eta_\epsilon$, for each $x \in E$ and $r \in (0,1]$
    \[
      4\epsilon^{-1}\nu(B_\epsilon(x,r))= \frac{T_\epsilon Z_{1-\sqrt\epsilon}}{Z_1} \tilde \nu\left(\tilde B_{\epsilon}\left(x,T_\epsilon U(r)\epsilon\right)\right)
    \]
    where $\tilde B_\epsilon (p,r) \subset (\tilde E, T_\epsilon \tilde\delta)$ is the closed ball of radius $r>0$ around $p$. Thus it follows that on the event $\mathcal{E}^\eta_\epsilon$ for each $x \in E$,
    \begin{equation}\label{eq:pinching_measures}
      \frac{1}{1+\eta}T_\epsilon\tilde\nu(\tilde B_{\epsilon}(x,(1-\eta)r)) \leq 4\epsilon^{-1}\nu(\bar B_\epsilon(x,r))\leq \frac{1}{1-\eta}T_\epsilon \tilde\nu(\tilde B_{\epsilon}(x,(1+\eta)r)) .
    \end{equation}
    Using the fact that $\P(\mathcal{E}^\eta_\epsilon) \rightarrow 1$ as $\epsilon \rightarrow 0$, Lemma \ref{lemma:tilde_S_close} and that $\eta>0$ is arbitrary, an easy pinching argument using \eqref{eq:pinching} and \eqref{eq:pinching_measures} shows that
    \begin{equation}\label{eq:tilde_E_close}
      \lim_{\epsilon \rightarrow 0}d_{GHP}((B_\epsilon(1,1),\epsilon^{-1}\delta, 4\epsilon^{-1}\nu), (\tilde B_\epsilon(1,1),T_\epsilon\tilde\delta, T_\epsilon \tilde \nu))=0
    \end{equation}
    almost surely. The theorem now follows from Lemma \ref{lemma:tilde_S_close}.
  \end{proof}

  \printbibliography
\end{document}